\documentclass[12pt,oneside,a4paper]{amsart}

\usepackage{amscd}
\usepackage{amssymb}
\usepackage{a4wide}
\usepackage{amstext}
\usepackage{amsthm}
\usepackage{mathrsfs}

\usepackage{latexsym}
\usepackage{xcolor}
\numberwithin{equation}{section}

\newcommand{\half}{{1/2}}

\newcommand{\RR}{\mathbb{R}}

 \DeclareMathOperator{\dist}{dist}

\DeclareMathOperator{\ospan}{\overline{Span}}

\renewcommand{\phi}{\varphi}

\newcommand{\vep}{\varepsilon}
\newcommand{\co}{\mathbb{C}}

\newcommand{\cp}{\mathbb{C_+}}
\newcommand{\rl}{\mathbb{R}}

\newcommand{\he}{\mathcal{H}(E)}

\renewcommand{\phi}{\varphi}

\newcommand{\A}{\mathcal{A}}
\newcommand{\T}{\mathcal{T}}
\renewcommand{\L}{\mathcal{L}}

\newcommand{\spa}{\operatorname{span}}
\newcommand{\Ker}{\operatorname{Ker}}

\newcommand{\tz}{\mathbb{T}}
\newcommand{\cm}{\mathbb{C_-}}
\newcommand{\rea}{{\rm Re}\,}
\newcommand{\ima}{{\rm Im}\,}

\newtheorem{Thm}{Theorem}[section]
\newtheorem{theorem}[Thm]{Theorem}
\newtheorem{lemma}[Thm]{Lemma}
\newtheorem{proposition}[Thm]{Proposition}

\newtheorem{remark}[Thm]{Remark}

\begin{document}
\sloppy
\title[Spectral synthesis in de Branges spaces]
{Spectral synthesis in de Branges spaces}

\author{Anton Baranov, Yurii Belov, Alexander Borichev}
\address{Anton Baranov,
\newline Department of Mathematics and Mechanics, St.~Petersburg State University, St.~Petersburg, Russia,
\newline National Research University Higher School of Economics, St.~Petersburg, Russia,
\newline {\tt anton.d.baranov@gmail.com}
\smallskip
\newline \phantom{x}\,\, Yurii Belov,
\newline Chebyshev Laboratory, St.~Petersburg State University, St. Petersburg, Russia,
\newline {\tt j\_b\_juri\_belov@mail.ru}
\smallskip
\newline \phantom{x}\,\, Alexander Borichev,
\newline I2M, Aix-Marseille Universit\'e, CNRS, Marseille, France
\newline {\tt alexander.borichev@math.cnrs.fr}
}
\thanks{A.~Baranov and Yu.~Belov were supported by Russian Science Foundation grant 14-21-00035.}

\begin{abstract}
We solve completely the spectral synthesis problem for reproducing
kernels in the de Branges spaces $\he$. Namely, we describe the de
Branges spaces $\he$ such that every complete and minimal system
of reproducing kernels $\{k_\lambda\}_{\lambda\in\Lambda}$ with complete biorthogonal
$\{g_\lambda\}_{\lambda\in\Lambda}$ admits the spectral synthesis, i.e.,  
$f \in \ospan \{(f, g_\lambda) k_\lambda:\lambda\in\Lambda\}$ for 
any $f$ in $\he$. 
Surprisingly, this property takes place only for two essentially
different classes of de Branges spaces: spaces with finite
spectral measure and spaces which are isomorphic to Fock-type
spaces of entire functions. The first class goes back to de
Branges himself, while special cases of de Branges
spaces of the second class appeared in the literature only
recently; we give a complete characterisation of this second class
in terms of the spectral data for $\he$. 
\end{abstract}

\maketitle


\section{Introduction and main results}

\subsection{De Branges spaces and systems of reproducing kernels}

The de Branges spaces of entire functions play a central r\^ole in various problems of spectral theory of second order ordinary differential operators 
and of harmonic analysis. The basic theory of these spaces was summarized by de Branges in \cite{br}. 
For more recent developments of this theory and its diverse applications see the works \cite{hv,os,mp,lag,mp1}; by no means is this list complete. The essence 
of the theory of de Branges spaces are their reproducing kernels.

Here, we study geometric properties of systems of reproducing kernels in the de Branges spaces. 
Let us recall a few notions. A system of vectors  $\{x_n\}_{n\in N}$ in a separable Hilbert space
$H$ is said to be {\it exact} if it is both {\it complete}
(i.e., $\ospan \{x_n\} = H$)
and {\it minimal}
(i.e., $\ospan \{x_n\}_{n\neq n_0} \neq H$ for any $n_0\in N$). 
Given an exact system there exists a unique {\it biorthogonal}
system $\{\tilde x_n\}_{n\in N}$ which satisfies the relation
$ (x_m, \tilde x_n) = \delta_{mn}$. Thus, to every element $x\in H$ 
we can associate its (formal) Fourier series
\begin{equation}
\label{rf}
x  \sim \sum_{n\in N}(x,\tilde x_n)x_n.
\end{equation}
This correspondence is one-to-one whenever no $x\in H\setminus\{0\}$ generates zero series, that is whenever 
the biorthogonal system $\{\tilde x_n\}_{n\in N}$ is also complete.
Such exact system  
is said to be a {\it Markushevich basis} (or {\it
$M$-basis}). A very natural property is the possibility to recover any vector $x\in H$ from its Fourier series:
$$
x\in \ospan \{(x,\tilde x_n) x_n\}.
$$
If this holds, then the system
$\{x_n\}_{n\in N}$ is called {\it a strong Markushevich basis} ({\it strong $M$-basis}), 
{\it a hereditarily complete system}
or one says that the system $\{x_n\}$ {\it admits the spectral synthesis},
see the operator theory motivation below. 
An equivalent definition of a strong $M$-basis is that
for any partition $N = N_1 \cup N_2$, $N_1 \cap N_2 =\emptyset$, of the index set $N$,
the mixed system
$$
\{x_n\}_{n\in N_1} \cup \{\tilde x_n\}_{n\in N_2}
$$
is complete in $H$. Furthermore, an $M$-basis which is not a strong
$M$-basis is called {\it a nonhereditarily complete system}.

In this paper we give a complete description of the de Branges spaces such that every $M$-basis of reproducing kernels is strong. 
There are several motivations for this problem (which essentially goes back to Nikolski):
\begin{itemize}
\begin{item}
{\it Strong $M$-bases of exponentials on an interval.}
This is a special case of the strong $M$-basis problem for reproducing kernels
since the Paley--Wiener space $\mathcal{P}W_a = \mathcal{F}L^2(-a,a)$
is a de Branges space and exponentials
correspond to reproducing kernels of $\mathcal{P}W_a$ via the Fourier
transform $\mathcal{F}$, see \cite{bbb}.
\end{item}
\begin{item}
{\it Spectral synthesis for a class of linear operators.}
Systems of reproducing kernels in de Branges spaces
appear (in an appropriate functional model) as eigenfunctions
of rank one perturbations of compact selfadjoint operators and the strong $M$-bases of
reproducing kernels correspond to the possibility of the spectral synthesis
(see \cite{bd} 
for more details).
\end{item}
\begin{item}
{\it Applications to differential operators.} Systems of
reproducing kernels correspond to eigenvectors of Schr\"odinger
operators via the Weil--Titchmarsh transform (see, e.g.,
\cite{mp}) or, more generally, of canonical systems of
differential equations.
\end{item}
\end{itemize}

A class of nonhereditarily complete systems of reproducing kernels was constructed in \cite{bbb}. 
If, on the contrary,
any exact system of reproducing kernels with the complete biorthogonal system
in a de Branges space is a strong $M$-basis we say that this de Branges space
has {\it strong $M$-basis property}.


\subsection{Description of de Branges spaces with strong $M$-basis property}

Now we state the main result of the paper. All necessary definitions from
de Branges spaces theory will be given in Section \ref{debr}. Each de Branges space
$\he$ is generated by some entire function $E$ of the Hermite--Biehler class.
On the other hand, the space $\he$ can be identified with
the space $\mathcal{H}$ of all entire functions of the form
$$
F(z) = A(z) \sum_n \frac{a_n \mu_n^{1/2}}{z-t_n}, \qquad \{a_n\}\in\ell^2,
$$
where
\begin{itemize}
\begin{item}
$T = \{t_n\}_{n\in N}$ is an increasing sequence such that $|t_n|\to\infty$,
$|n| \to\infty$, $N=\mathbb Z$ or $\mathbb Z_+$ or $\mathbb Z_-$; 
\end{item}
\begin{item}
$\mu = \sum_n \mu_n \delta_{t_n}$ is a positive measure on $\mathbb{R}$
satisfying $\sum_n \frac{\mu_n}{t_n^2+1}<\infty$;
\end{item}
\begin{item}
$A$ is an entire function with zero set $T$ (all zeros are simple)
which is real on $\rl$;
\end{item}
\begin{item}
the norm of $F$ is defined
as $\|F\|_{\mathcal{H}} = \|\{a_n\}\|_{\ell^2}$.
\end{item}
\end{itemize}
For the details of this identification see Section \ref{debr}. We will call the
pair $(T, \mu)$ the {\it spectral data} for the space $\he$. For a special
class of de Branges spaces (regular spaces) which appear in the de Branges
inverse spectral theorem, $\mu$ is the spectral measure of
 the operator associated with the canonical system.

It should be noted that this point of view on de Branges spaces was used in
\cite{BMS, BMS1}, where a more general model of spaces of entire functions
with Riesz bases of reproducing kernels was developed.

In what follows we will always assume for simplicity that $0\notin T$.
We will need the following two conditions on the sequence $T$.
\medskip
\\
{\bf Definition.} We will say that the sequence $T$ is
\begin{itemize}
\begin{item}
{\it lacunary} (or {\it Hadamard lacunary}) if
$$
\liminf_{t_n \to \infty} \frac{t_{n+1}}{t_n}>1, \qquad
\liminf_{t_n \to -\infty} \frac{|t_{n}|}{|t_{n+1}|}>1.
$$
Equivalently, this means that for some $\delta>0$
we have $d_n:= t_{n+1}-t_n \ge  \delta |t_n|$.
\end{item}
\begin{item}
{\it power separated} if there exist $c,\, N>0$ such that
$$
d_n= t_{n+1}-t_n \ge c|t_n|^{-N}.
$$
\end{item}
\end{itemize}

The absence of strong $M$-basis property was
previously known only for some special examples of de Branges
spaces (including the Paley--Wiener space). It was shown in
\cite{bbb} that if $T$ is power separated and $d_n = o(|t_n|)$,
$|n|\to\infty$, then {\it one can choose } $\mu_n$
such that the corresponding de Branges space $\he$ does not have
strong $M$-basis property. We are now in position to give a
complete characterisation of such de Branges spaces in terms of
their spectral data.

\begin{theorem}
\label{mainn}
Let $\he$ be a de Branges space with the spectral data $(T, \mu)$.
Then $\he$ has the strong $M$-basis property
if and only if one of the following conditions holds:
\smallskip

{\rm (i)} $\sum_n \mu_n < \infty$\textup;
\smallskip

{\rm (ii)} The sequence $\{t_n\}$ is lacunary and, for some $C>0$ and any $n$,
\begin{equation}
\label{mulacunary}
\sum_{|t_k|\leq|t_n|}\mu_k+t^2_n\sum_{|t_k|>|t_n|}\frac{\mu_k}{t^2_k}\le C \mu_n.
\end{equation}
\end{theorem}

Thus, there exist two distinct classes of de Branges
spaces with strong $M$-basis property. It seems that there are deep reasons
for this property which are essentially different in these two cases:

\begin{itemize}
\begin{item}
For the case $\sum_n \mu_n<\infty$ there exists an operator theory
explanation. Passing to the model of rank one perturbations of
selfadjoint  operators \cite{bd} we find ourselves in the case of {\it weak
perturbations} in the sense of Macaev. It is known
that this class of perturbations is more regular than the general
rank one perturbations.
\end{item}

\begin{item}
Perturbations of the form \eqref{mulacunary} are, on contrary,
large, but the spectrum is lacunary. It turns out that in this
case de Branges space coincides (as a set with equivalence of
norms) with a Fock-type space.
\end{item}
\end{itemize}

Note also that in the case (ii) any exact system of reproducing
kernels in $\he$ is automatically an $M$-basis (see \cite[Theorem 1.2]{bb}),
whereas in the case (i) there always exist exact systems of reproducing
kernels with incomplete biorthogonal system
(see \cite[Theorem 1.1]{bb} or  Proposition \ref{bior}).


\subsection{De Branges spaces with strong $M$-basis property and Fock-type spaces}

Let $\phi: [0, \infty) \to (0, \infty)$ be a measurable function. With each $\phi$ we associate a {\it radial Fock-type space}
(or a {\it Bargmann--Fock space})
$$
\mathcal{F}_\varphi = \Big\{F\ \text{entire}\,:
\|F\|_{\mathcal{F}_\varphi}^2 := \int_\co |f(z)|^2e^{-\phi(|z|)} dm(z) <\infty
\Big\}.
$$
Here $m$ stands for the area Lebesgue measure. The classical Fock space corresponds to $\phi(r)=\pi r^2$.

It is known that some Fock-type spaces with slowly growing $\phi$
(e.g., $\phi(r) = (\log r)^\gamma$, $\gamma \in (1,2]$)
have Riesz bases of reproducing kernels corresponding to real points and, thus,
can be realized as de Branges spaces with equivalence of norms \cite{bl} (whereas the classical
Fock space has no Riesz bases of reproducing kernels).
Surprisingly, it turns out that the class of de Branges spaces which can be realized
as Fock-type spaces coincides exactly with the class of de Branges spaces
with strong $M$-basis property from Theorem \ref{mainn}, (ii).

\begin{theorem}
\label{fock1}
Let $\he$ be a de Branges space with the spectral data $(T, \mu)$.
Then the following conditions are equivalent\textup:
\begin{enumerate}
\begin{item}
There exists a Fock-type space $\mathcal{F}_\varphi$ such that
$\he=\mathcal{F}_\varphi$\textup;
\end{item}
\begin{item}
The operator $R_\theta: f(z)\mapsto f(e^{i\theta}z)$
is a bounded invertible operator in $\he$ for all \textup(some\textup)
$\theta\in(0,\pi)$\textup;
\end{item}
\begin{item}
The sequence $T$ is lacunary and condition \eqref{mulacunary} holds.
\end{item}
\end{enumerate}
\end{theorem}
\medskip

Let us also mention here an important characteristic property of de Branges spaces: each Hilbert space of entire functions possessing two orthogonal bases of reproducing kernels is a de Branges space \cite{BMS1}. 


\subsection{Nonhereditarily complete systems with infinite defect}

Assume that for some de Branges space $\he$ strong
$M$-basis property fails. In view of the applications to the
spectral synthesis of linear operators, it is an important problem
to determine whether the codimension of every mixed system of
reproducing kernels and their biorthogonals is finite or
infinite.

Let $\{k_\lambda\}_{\lambda\in \Lambda}$ be an exact system of
reproducing kernels with a complete biorthogonal system
$\{g_\lambda\}_{\lambda\in \Lambda}$ (see Subsection \ref{reduc}
for the definition of $g_\lambda$). For a partition $\Lambda =
\Lambda_1\cup \Lambda_2$, define the defect of the corresponding
mixed system as
$$
{\rm def}\,(\Lambda_1, \Lambda_2):=
{\rm dim}\,(\{g_\lambda\}_{\lambda\in \Lambda_1} \cup
\{k_\lambda\}_{\lambda\in \Lambda_2})^{\perp}.
$$
We also put
$$
{\rm def}(\Lambda)=\sup\{{\rm def}(\Lambda_1, \Lambda_2):\,
\Lambda = \Lambda_1\cup \Lambda_2\},
$$
$$
{\rm def}\big(\he\big)=\sup\{{\rm def}(\Lambda):
\{k_\lambda\}_{\lambda\in\Lambda} \text{ is }M\text{-basis}\}.
$$
It turns out that one can construct examples of $M$-bases of reproducing kernels
with large or even infinite defect.

\begin{theorem}
\label{mdim}
Let $\he$ be a de Branges space with the spectral data $(T, \mu)$,  \smallskip
$\sum_n \mu_n =\infty$.

\textup{1}. If for some $N\in\mathbb{N}$ there exists a subsequence $t_{n_k}$ of $T$
such that $\sum_k t_{n_k}^{2N-2} \mu_{n_k}<\infty$, then ${\rm def}(\he)\geq N$ \textup{(}moreover, there exist an $M$-basis
of reproducing kernels $\{k_\lambda\}_{\lambda\in \Lambda}$
in $\he$ such that ${\rm def}\,(\Lambda_1, \Lambda_2) = N$
for some partition $\Lambda = \Lambda_1\cup \Lambda_2$\textup{)}.
\smallskip

\textup{2}. Let $T$ be a power separated sequence. Then the following are equivalent:

\textup{(i)} ${\rm def}\big(\he\big) =\infty$\textup{;}


\textup{(ii)} $\inf_{n} \mu_n |t_n|^N =0$  for any $N>0$.
\end{theorem}

This theorem gives, for a wide class of spectral data,
a necessary and sufficient condition for the existence
of $M$-bases of reproducing kernels with  arbitrarily large defects.
Its proof does not provide, however, an example of an $M$-basis
with infinite defect ${\rm def}(\Lambda_1, \Lambda_2)$.

\begin{theorem}
\label{indim}
For any increasing sequence $T =\{t_n\}$ with $|t_n| \to \infty$, $|n| \to \infty$,
there exists a measure $\mu$ such that 
the de Branges space with the spectral
data $(T, \mu)$
contains an $M$-basis
of reproducing kernels $\{k_\lambda\}_{\lambda\in \Lambda}$
such that ${\rm def}\,(\Lambda_1, \Lambda_2) = \infty$
for some partition $\Lambda = \Lambda_1\cup \Lambda_2$.
\end{theorem}


\subsection{Spectral theory of rank one perturbations of selfadjoint operators}
\label{spect}
A continuous operator $\L$ in a Banach (or a Fr\'echet) space is said to 
admit the {\it spectral synthesis} (in the sense of Wermer) if
any $\L$-invariant closed linear subspace $M$ 
is spanned by the root vectors it contains.
Wermer \cite{wer} showed that any compact normal 
operator admits the spectral synthesis. 
Neither normality nor compactness alone is sufficient.
The first example of a compact operator $\L$ such that both 
$\L$ and $\L^*$ have complete sets of eigenvectors, but $\L$ admits no spectral synthesis 
was given by Hamburger \cite{hamb}, who constructed
a compact operator with a complete set of eigenvectors, whose restriction
to an invariant subspace is a nonzero Volterra operator.
Further examples of operators which do not admit the spectral synthesis
were obtained by Nikolski \cite{nik69} and Markus \cite{markus}. 
In particular, it was shown in \cite[Theorem 4.1]{markus} that
a compact operator admits the spectral synthesis if and only if
its root vectors form a strong $M$-basis.

In \cite{bd} a functional model is constructed for rank one perturbations
of compact self-adjoint operators (see \cite{bd} for an extensive survey of similar models).
Let $\A$ be a compact self-adjoint operator in a Hilbert space $H$ such that $\Ker \A = 0$
and the spectrum of $\A$ is simple. For $a, b\in H$, let $\L = \A + a\otimes b$, 
where $(a\otimes b)x = (x, b)a$.  Any such perturbation is unitarily equivalent to
a certain model operator acting in a de Branges space $\he$. The model operator 
(unbounded, defined on some appropriate domain which is dense in $\he$) is given by 
$$
\T F = zF - c_F G,
$$ 
where $G$ is an entire function such that $G \notin \he$, but $G/(\cdot-\lambda) \in \he$ when 
$G(\lambda) = 0$. Here $c_F$ is some constant depending on $F$ (see \cite[Theorem 4.4]{bd}
or \cite[Section 4]{bbby} for details). Conversely, any pair $(E, G)$ as above corresponds 
to some rank one perturbation of a compact self-adjoint operator
whose spectrum is the zero set of the function $E+E^*$, where $E^*(z) = \overline{E(\overline z)}$. 
Clearly, eigenvectors
of $\T$ are of the form $g_\lambda = G/(\cdot-\lambda)$, where $G(\lambda) = 0$, i.e., 
they are  elements  
of the system biorthogonal to a system of reproducing kernels in $\he$
(the corresponding reproducing kernels are eigenvectors of the adjoint operator $\T^*$).  

Combining Theorems \ref{mdim} and \ref{indim} with the above model
we obtain a series of striking examples: any compact self-adjoint operator can be turned 
by a rank one perturbation into an operator for which the spectral synthesis
fails up to a finite or even an infinite defect. Denote by $\mathcal{E}(\L)$
the set of eigenvetors of the operator $\L$.

\begin{theorem}
\label{sps}
For any compact self-adjoint operator $\A$ \textup(in some Hilbert space $H$\textup) 
with simple spectrum
and trivial kernel there exists a rank one perturbation $\L$ of $\A$ with real spectrum
such that both $\mathcal{E}(\L)$ and $\mathcal{E}(\L^*)$ are complete in $H$, 
but $\L$ does not admit the spectral synthesis. 
Moreover, for any $N\in\mathbb{N} \cup \{\infty\}$, the rank one perturbation may be chosen so that
for some $\L$-invariant subspace $M$ we have
$$
{\rm dim}\,\big(M\ominus \spa \{\mathcal{E}(\L)\cap M\}\big) = N.
$$
\end{theorem}  

Indeed, let $\A$ be a compact self-adjoint  operator with simple spectrum $\{s_n\}$, $s_n\ne 0$.
Put $T = \{t_n\}$, $t_n = 1/s_n$. By Theorems \ref{mdim} and \ref{indim} 
one can construct a measure $\mu = \sum_n \mu_n \delta_{t_n}$ 
such that the de Branges space $\he$
with the spectral data $(T, \mu)$ contains a complete and minimal system 
of reproducing kernels $\{k_\lambda\}_{\lambda\in \Lambda}$ with the
generating function $G$, which is not a strong $M$-basis.
Moreover, we can achieve that ${\rm def}\,(\Lambda_1, \Lambda_2) =N$ for some 
system of the form $\{g_\lambda\}_{\lambda\in \Lambda_1} \cup
\{k_\lambda\}_{\lambda\in \Lambda_2}$.
By the functional model, 
there exists a rank one perturbation $\L$ of $\A$ whose eigenvectors may be identified 
(via a unitary equivalence) with $\{g_\lambda\}_{\lambda\in \Lambda}$.
Clearly, $M = \{k_\lambda:\, \lambda\in \Lambda_2\}^\perp$ is an $\L$-invariant subspace, 
and $g_\lambda \in M$ for $\lambda\in \Lambda_1$, but the complement of
$\{g_\lambda\}_{\lambda\in \Lambda_1}$ in $M$ is of dimension $N$.

\subsection{Strong $M$-bases in the general setting}

It is not a completely trivial problem to produce a nonhereditarily complete system in a separable Hilbert space.
First explicit examples of $M$-bases which are not strong were
constructed by Markus \cite{markus} in 1970. Later, Nikolski,
Dovbysh and Sudakov studied in detail the structure of
nonhereditarily complete systems in a Hilbert space and produced
many further examples. One more series of examples was constructed
by Larson and Wogen \cite{lar}, Azoff and Shehada \cite{az} and
Katavolos, Lambrou, and Papadakis \cite{kat}. These systems  are
obtained by an application of a three-diagonal matrix to the
standard orthonormal basis in $\ell^2$.


Clearly, the property of being a strong $M$-basis is necessary
for the existence of a linear summation method for the Fourier series \eqref{rf}.
It is far from being sufficient. E.g., one may deduce from the results
of an interesting paper \cite{kat} that there are strong $M$-bases with
the following property: there exist two vectors $h_1$ and $h_2$ in $H$
such that one cannot approximate $h_1$ by a linear combination of
$(h_1,\tilde x_n) x_n$ and $h_2$ by a linear combination of
$(h_2,\tilde x_n) x_n$ with the same coefficients.

\medskip

{\bf Organization of the paper.} The paper is organized as
follows. Main Theorem \ref{mainn} is proved in Sections
\ref{strat}--\ref{proofcod}. A detailed outline of
the proof is given in Subsection \ref{outline}. Theorem
\ref{fock1} about de Branges spaces coinciding with Fock-type
spaces is proved in Section \ref{fock}. Examples of $M$-bases of
kernels with large defects are constructed in Section
\ref{infdim}.
\smallskip

{\bf Notations.} Throughout this paper the notation $U(x)\lesssim V(x)$ (or, equivalently,
$V(x)\gtrsim U(x)$) means that there is a constant $C$ such that
$U(x)\leq CV(x)$ holds for all $x$ in the set in question, which may
be a Hilbert space, a set of complex numbers, or a suitable index
set. We write $U(x)\asymp V(x)$ if both $U(x)\lesssim V(x)$ and
$V(x)\lesssim U(x)$. We write $u_n \ll v_n$ (usually in the context of sequences)
if $u_n =o (v_n)$ as $n\to \infty$. For an entire function $F$ we denote by
$\mathcal{Z}_F$ the set of its zeros. For a finite
set $Y$ we denote by $\#Y$ the number of its elements.
\medskip

{\bf Acknowledgements.}
The authors are grateful to Dmitry Yakubovich 
for many useful remarks 
and to Mikhail Sodin 
for his constructive comments that helped to improve the presentation of the paper.
\bigskip


\section{Preliminaries on de Branges spaces and Clark measures}
\label{debr}

\subsection{De Branges spaces}
\label{debr0}
An entire function $E$ is said to be in the Hermite--Biehler class if
$$
|E(z)| > |E^*(z)|,  \qquad z\in {\mathbb{C}_+},
$$
where $E^* (z) = \overline {E (\overline z)}$. 
With any such function we associate the {\it de Branges space}
$\mathcal{H} (E) $ which consists of all entire functions
$F$ such that $F/E$ and $F^*/E$ restricted to $\mathbb{C_+}$ belong
to the Hardy space $H^2=H^2(\mathbb{C_+})$.
The inner product in $\he$ is given by
$$
( F,G)_E = \int_\rl \frac{F(t)\overline{G(t)}}{|E(t)|^2} \,dt.
$$
The reproducing kernel of the de Branges space ${\mathcal H} (E)$
corresponding to the point $w\in \mathbb{C}$ is given by
\begin{equation}
\label{repr}
k_w(z)=\frac{\overline{E(w)} E(z) - \overline{E^*(w)} E^*(z)}
{2\pi i(\overline w-z)} =
\frac{\overline{A(w)} B(z) -\overline{B(w)}A(z)}{\pi(z-\overline w)},
\end{equation}
where the entire functions $A$ and $B$ are defined by $A = \frac{E+E^*}{2}$,
$B=\frac{E^*-E}{2i}$, so that $A$ and $B$ are real on $\mathbb{R}$
and $E=A - iB$.

There exists an equivalent axiomatic description of de Branges
spaces \cite[Theorem 23]{br}: any reproducing kernel Hilbert space
of entire functions $\mathcal{H}$ such that the mapping $F\mapsto
F^*$ preserves the norm in $\mathcal{H}$ and the mapping $F\mapsto
\frac{z-\overline w}{z-w}F(z)$ is an isometry in $\mathcal{H}$
whenever $w\in\co\setminus\mathbb{R}$, $F(w) = 0$, is of the form
$\he$ for some $E$ in the Hermite--Biehler class.

From now on we restrict ourselves to the case when $E$
has no real zeros and, correspondingly, $\he$ has no common zeros 
in the complex plane.

The space $\mathcal{H}(E)$ is essentially defined (up to a canonical isomorphism)
by the function $\Theta_E = E^*/E$ which is inner in $\cp$ and
meromorphic in $\co$: the mapping $F\mapsto F/E$
is a unitary operator from $\mathcal{H}(E)$
onto the subspace $K_{\Theta_E} = H^2\ominus\Theta_E H^2$ of the Hardy space
$H^2$ known as a {\it model subspace}.

However, it is often useful to think about de Branges spaces not in terms of
the zeros of $E$, but in terms of the zeros of $A$ (or $B$)
and some associated measure supported by $\mathcal{Z}_A$.
It is a crucial property of de Branges spaces that there
exists a family of orthogonal bases of reproducing kernels
corresponding to real points \cite[Theorem 22]{br}.
Namely, for any $\alpha \in \tz = \{z\in \co: |z|=1\}$ consider the set
$T_\alpha$ of points $t_{\alpha, n} \in \mathbb{R}$ such that
$\Theta_E(t_{\alpha, n}) = \alpha$.
Then the system of reproducing kernels $\{k_{t_{\alpha, n}}\}$
is an orthogonal basis for $\he$ for each $\alpha\in \tz$
except, may be, one ($\alpha$ is an exceptional value if and only if
$E_\alpha:= (\alpha E - E^*)/2 \in \he$).

The points $t_{\alpha, n}$ may be also obtained via the so-called
{\it phase function} for $E$. For an Hermite--Biehler function $E$
there exists a smooth increasing function $\phi$ (called the phase function of $E$)
such that $E(t)e^{i\phi(t)} \in \rl$, $t\in \rl$ (equivalently, $2\phi$
is a branch of the argument of $\Theta_E$ on $\rl$).
Clearly, $\phi$ is uniquely defined up to a constant $\pi l$, $l\in \mathbb{Z}$.
Then $t_{\alpha, n}$ are the solutions of the equation
$\phi(t_{\alpha, n}) = \frac12 \arg \alpha + \pi n$. Note that
$t_{\alpha, n}$ may exist for all $n\in \mathbb{Z}$, for $n\in [n_1,\infty)$
or for $n\in (-\infty, n_2]$, where $n_1$, $n_2 \in \mathbb{Z}$. The case
where there is only finite number of solutions $t_{\alpha, n}$
corresponds to finite-dimensional de Branges spaces where
the strong $M$-basis problem is trivial.

Thus, for all $\alpha \in \tz$ except the possible exceptional value,
the system $\big\{\frac{E_\alpha(z)}{z-t_{\alpha, n}}\big\}$
is an orthogonal basis in $\he$
and any function $F\in\he$ admits the expansion
$$
F(z) = E_\alpha(z) \sum_n \frac{a_n \mu_{\alpha,n}^{1/2}}{z-t_{\alpha,n}},
$$
where $\{a_n\} \in \ell^2$ and $\mu_{\alpha,n} = \pi^2\big\|\frac{E_\alpha(z)}
{z-t_{\alpha, n}}\big\|_E^{-2}$. Also, by \eqref{repr},
$$
\frac{E_\alpha(z)}{z-t_{\alpha, n}} =
-\frac{\pi i}{E(t_{\alpha, n})}k_{t_{\alpha, n}}(z), \qquad
\mu_{\alpha,n} = \frac{\pi}{\phi'(t_{\alpha,n})}
$$
(recall that $\|k_t\|^2_E = |E(t)|^2 \phi'(t)/\pi$, $t\in \rl$).

Let $\mu_\alpha = \sum_n \mu_{\alpha,n} \delta_{t_{\alpha,n}}$.
We call the elements of the family $(T_\alpha, \mu_\alpha)$
(excluding the possible exceptional value of $\alpha$) the {\it spectral
data} for the de Branges space $\he$.
In what follows it will be often convenient to pass from one
data (i.e., orthogonal basis of reproducing kernels)
to another; this method played also a crucial role in \cite{bbb}.


\subsection{An alternative approach to de Branges spaces}
\label{alter}

In what follows we will always assume that $A \notin \he$ (in other words, $(T_{-1},\mu_{-1})$ is a spectral data for $\he$).
Let $T=T_{-1}=\{t_n\}$ be the zero set of $A$, that is,
the set $\{t:\, \Theta_E(t) = -1\}$. Then
$\big\{\frac{A(z)}{z-t_n}\big\}$ is an orthogonal basis in $\he$
and any function $F\in\he$ admits the expansion
\begin{equation}
\label{cauch}
F(z) = A(z) \sum_n \frac{a_n \mu_n^{1/2}}{z-t_n},
\end{equation}
where $\sum_n \frac{\mu_n}{t_n^2+1}<\infty$
and $\|\{a_n\}\|_{\ell^2} = \|F\|_E/\pi <\infty$.
Conversely, this expansion may be taken as a definition of the de Branges space.
Starting from any increasing sequence $T = \{t_n\}_{n\in N}$, $|t_n|\to\infty$
as $|n| \to\infty$, and a measure
$\mu = \sum_n \mu_n \delta_{t_n}$ satisfying $\sum_n \frac{\mu_n}{t_n^2+1}<\infty$
we may consider the space of all entire functions of the form \eqref{cauch}
with the norm $\|F\| = \pi\|\{a_n\}\|_{\ell^2}$. Here $A$ can be taken to be any entire
function with simple zeros which is real on $\rl$ and whose zero set
coincides with $T$.
Then the corresponding space of entire functions is a de Branges space $\he$
for some $E$ with $\frac{E+E^*}{2}=A$. The corresponding function $B$
will be given by
\begin{equation}
\label{herg}
\frac{B(z)}{A(z)} = r + \frac 1{\pi}\sum_n \bigg(\frac{1}{t_n-z}
- \frac{t_n}{t_n^2+1}\bigg) \mu_n,
\end{equation}
where $r$ is some real number.

In \cite{BMS} a more general model for reproducing kernel Hilbert spaces
of entire functions with a Riesz basis of reproducing kernels was developed.
Let $\mathcal{H}$ be a reproducing kernel Hilbert space
of entire functions such that in $\mathcal{H}$ there exists
a Riesz basis of reproducing kernels $\{k_{w_n}^{\mathcal{H}}\}_{w_n \in W}$
and such that $\mathcal{H}$ is closed under division by zeros:
if $F\in \mathcal{H}$ and $F(w) = 0$, $w\in \co$, then $\frac{F(z)}{z-w} \in \mathcal{H}$.
Then there exists a positive sequence $\mu_n$ such that
$\sum_n \frac{\mu_n}{|w_n|^2+1} < \infty$ and an entire function
$A$ (which has only simple zeros exactly on $W$) such that
$\mathcal{H}$ coincides with the space of the  entire functions of the form
$F(z) = A(z) \sum_n \frac{a_n \mu_n^{1/2}}{z-w_n}$ and $\|F\|_{\mathcal{H}} \asymp
\|\{a_n\}\|_{\ell^2}$.

The structure of such spaces is determined by the data $(W, \mu)$,
and does not depend on the choice of $A$. If $A_1= AS$ for a nonvanishing entire function
$S$, then the mapping $F\mapsto SF$ is an isomorphism of the corresponding space.
Therefore, it makes sense to consider the space of meromorphic functions
(the Cauchy transform)
$$
\mathcal{H}(W, \mu)= \bigg\{f(z) = \sum_n \frac{a_n \mu_n^{1/2}}{z-w_n}: \,
\{a_n\}\in \ell^2, \|f\|_{\mathcal{H}(W, \mu)}:=\pi\|\{a_n\}\|_{\ell^2} \bigg\}
$$
De Branges spaces correspond to the case when all $w_n$ are real.


\subsection{Remarks on Clark measures}
\label{clar}

Recall that all spectral data $(T_\alpha, \mu_\alpha)$
may be obtained as follows: let $\mu_\alpha$
be the measure from the Herglotz representation
$$
\rea \frac{\alpha E(z) + E^*(z)}{\alpha E(z) - E^*(z)} = p_\alpha y 
+ \frac{y}{\pi}\int_\mathbb{R}\frac{d\mu_\alpha(t)}{|t-z|^2},
$$
where $z= x +iy$, $p_\alpha\ge 0$. 
Then $\mu_\alpha$ is an atomic measure supported by the zero set $T_\alpha$
of the function $\alpha E - E^*$. Measures $\mu_\alpha$ are often called
{\it Clark measures} after the seminal paper \cite{cl} though in the de Branges
space context they were introduced much earlier by de Branges himself.
\smallskip

Let us mention some well-known properties of the measures $\mu_\alpha$:
\smallskip

1. The system $\{k_t\}_{t\in T_\alpha}$ is an orthogonal basis in $\he$
for all $\alpha \in \tz$ except, possibly, one; $\alpha$ is an
exceptional value if an only if $p_\alpha>0$ (or, equivalently,
$\alpha E - E^* \in \he$) (see \cite[Problem 89]{br}).
\smallskip

2. $\mu_\alpha(\{t\}) = 2\pi |\Theta_E'(t)|^{-1} =
|E(t)|^2 \|k_t\|^{-2}_E$, $t\in T_\alpha$, whence the embedding of
$E^{-1} \he$ into $L^2(\mu_\alpha)$ is a unitary operator (only a co-isometry for the
exceptional value of $\alpha$).
\smallskip

3. If there exists an exceptional value of $\alpha$, then  $\mu_\beta(\rl)<\infty$
for any $\beta\in \tz$, $\beta\ne \alpha$ (indeed, the function $(\alpha E-E^*)/E$
is a nonzero constant on $T_\beta$ and belongs to $L^2(\mu_\beta)$).
\smallskip

4. If $\mu_\beta(\rl)<\infty$ for some $\beta\in \tz$, then there
exists an exceptional value $\alpha$ such that $\alpha E - E^* \in \he$.
Indeed, set $\tilde E=\beta E=\tilde A-i\tilde B$. If $\tilde A\notin \mathcal H(\tilde E)=\he$, then 
$\tilde B/\tilde A-\sum \mu_n/(\cdot -t_n)=c\in\mathbb R$, with $(t_n)=\mathcal Z_{\tilde A}$, $(\mu_n)=\mu_\beta$
and hence,
$(1+ic)/(1-ic)\tilde E-(\tilde E)^*\in\he$.
\smallskip

The masses of the Clark measures are determined by the equality $|\Theta_E'| = 2\phi'$ and,
therefore,
the estimates of the derivatives of meromorphic inner functions are of importance here.
Since $2A = E(1+\Theta_E)$ and $2iB =E( \Theta_E -1)$, it follows from
equation \eqref{herg} that
\begin{equation}
\label{use}
|\Theta_E'(t)| = \bigg|
i +r + \sum_n \frac{\mu_n}{\pi} \bigg(\frac{1}{t_n -t} -\frac{1}{t_n}\bigg)
\bigg|^{-2} \sum_n \frac{2 \mu_n}{\pi(t_n -t)^2}, \qquad t\in \mathbb{R},
\end{equation}
a formula which proves to be useful in what follows.


\section{Plan of the proof of Theorem \ref{mainn}}
\label{strat}

In this section we discuss the first step of the proof (reduction of the strong
$M$-basis problem to a system of interpolation equations)
and give a detailed outline of the proof.


\subsection{Reduction to an interpolation problem}
\label{reduc}

We will use the following criterion for being a strong $M$-basis:
{\it Let $\{x_n\}_{n\in N}$ be an $M$-basis in a Hilbert space $H$.
Then $\{x_n\}_{n\in N}$ is a strong $M$-basis if and only if for any
$h,\, \tilde h \in H$ such that $(h, x_n)\cdot(\tilde x_n, \tilde h) =0$
we have $(h,\tilde h) = 0$} (see, e.g., \cite{kat}).  Indeed, it is an obvious
reformulation of the definition of a strong $M$-basis that for any partition
$N=N_1 \cup N_2$, $\ospan\{x_n: n\in N_1\} = \{\tilde x_n: n\in N_2\}^\perp$.
The condition $(h, x_n)\cdot(\tilde x_n, \tilde h) =0$ means that for some partition $N=N_1 \cup N_2$ we have
$h \in \{x_n: n\in N_1\}^\perp$
and $\tilde h \in \{\tilde x_n: n\in N_2\}^\perp$,
whence $(h,\tilde h) = 0$. The converse implication is analogous.

We will apply this criterion to an $M$-basis of reproducing kernels
$\{k_\lambda\}_{\lambda\in \Lambda}$. Recall that its biorthogonal
system is given by $\Big\{\frac{G}{G'(\lambda)(\cdot-\lambda)}\Big\}$, where
$G$ is the so-called {\it generating function} of the set $\Lambda$, that is, an entire function with simple zeros whose zero set
coincides with $\Lambda$, $G\notin\he$ (moreover, $HG\in\he$ for an entire function $H$ implies that $H=0$), but $g_\lambda :=
\frac{G}{\cdot-\lambda} \in \he$ for any $\lambda\in \Lambda$.
In what follows we will always omit the normalizing factor $G'(\lambda)$
and say that $\{g_\lambda\}$ is the system biorthogonal to $\{k_\lambda\}$.
Then the system $\{k_\lambda\}_{\lambda\in \Lambda}$
is a strong $M$-basis if and only if any two vectors $h$, $\tilde h$ such that
\begin{equation}
(h, k_\lambda)\cdot(g_\lambda, \tilde h) = 0,\qquad \lambda\in\Lambda,
\label{ab}
\end{equation}
are orthogonal, $(h, \tilde h)=0$.

Without loss of generality (passing if necessary to some other spectral
data $(T_\alpha, \mu_\alpha)$) we can assume that
$\Lambda\cap T=\emptyset$. Then we rewrite condition \eqref{ab}.
Consider the expansions of the vectors $h$, $\tilde h$ with respect to
the de Branges orthonormal basis $\frac{k_{t_n}(z)}{\|k_{t_n}\|_E}
= \frac{(-1)^n\mu_n^{1/2}}{\pi}\cdot \frac{A(z)}{z-t_n}$,
$$
  h(z)=\pi\sum_{n}\overline{a_n}(-1)^n\cdot\frac{k_{t_n}(z)}{\|k_{t_n}\|}=
  A(z)\sum_{n}\frac{\overline{a_n}\mu^{1/2}_n}{z-t_n},\qquad \{a_n\}\in\ell^2,
$$
$$
  \tilde h(z)=\pi\sum_{n}\overline{b_n}(-1)^n\cdot\frac{k_{t_n}(z)}{\|k_{t_n}\|} =
  A(z)\sum_{n}\frac{\overline{b_n}\mu^{1/2}_n}{z-t_n}, \qquad \{b_n\}\in\ell^2.
$$
Then $(h, \tilde h)=\pi^2\sum_n\overline{a_n}b_n$.
Equation \eqref{ab} means that
there exists a partition $\Lambda=\Lambda_1\cup\Lambda_2$,
$\Lambda_1\cap\Lambda_2 = \emptyset$, such that
$(g_\lambda,\tilde h)=0$ for $\lambda\in\Lambda_1$
and $(h,k_\lambda)=0$ for $\lambda\in\Lambda_2$.
The second of these equalities means simply that
$h(\lambda) =0$, $\lambda\in\Lambda_2$, while
the first one may be rewritten as
$$
  \biggl{(}\frac{G(z)}{z-\lambda}, \tilde h\biggr{)}=
 \pi^2 \sum_n\frac{b_nG(t_n)}{A'(t_n)\mu^{1/2}_n(t_n-\lambda)}=0, \qquad \lambda\in\Lambda_1.
$$
Hence, we have a system of interpolation equations
\begin{equation}
 \label{main1}
  A(z)\sum_{n}\frac{b_nG(t_n)}{\mu^{1/2}_nA'(t_n)(z-t_n)}={G_1(z)S_1(z)},
\end{equation}
\begin{equation}
  \label{main2}
  A(z)\sum_{n}\frac{\overline{a_n}\mu^{1/2}_n}{z-t_n}={G_2(z)S_2(z)},
\end{equation}
where $G_1, G_2$ are \itshape some \normalfont entire functions
with simple zeros and zero sets $\Lambda_1,\Lambda_2$ respectively,
such that $G=G_1G_2$, while $S_1$ and $S_2$ are some entire functions.
Since $G$ is the generating function of an exact system of reproducing kernels,
we have $G/G^*=B_1/B_2$  for some Blaschke products $B_1, B_2$.
Next, we can assume that $G_1/G_1^*$, $G_2/G_2^*$ are also chosen
to be the ratios of two Blaschke products.

It should be mentioned that the above argument includes the case when
the biorthogonal system $\{g_\lambda\}_{\lambda\in \Lambda}$  is incomplete:
in this case we take $G_1 = G$ and $G_2=1$, and equation \eqref{main2}
becomes trivial.

Thus, we have proved the following proposition.

\begin{proposition}
\label{firstred}
Let $\he$ be a de Branges space. Then
$\he$ does not have strong $M$-basis property if and only if there exists an $M$-basis
$\{k_\lambda\}_{\lambda\in \Lambda}$ of reproducing kernels in $\he$
with the generating function $G$ such that for some partition
$\Lambda=\Lambda_1\cup\Lambda_2$, $\Lambda_1\cap\Lambda_2 = \emptyset$,
and for some spectral data $(T, \mu)$ for $\he$ such that $\Lambda_1\cap T = \emptyset$,
there exist sequences $\{a_n\},\, \{b_n\} \in\ell^2$ such that
$\sum_n \overline a_n b_n\ne 0$ and equations
\eqref{main1}--\eqref{main2} are satisfied for some entire functions
$S_1$ and $S_2$.
\end{proposition}

Comparing the values at the points $t_n$ in \eqref{main1}--\eqref{main2}
we obtain the following important relation:
if $S=S_1S_2$ then
\begin{equation}
S(t_n) = A'(t_n) \overline a_n b_n.
\label{st12}
\end{equation}

\begin{remark}
\label{st14}
{\rm 1. It is worth mentioning that if we want to find an $M$-basis
of reproducing kernels which is not strong, then it is sufficient to
find  two non-orthogonal sequences $\{a_n\}$ and $\{b_n\}$ which satisfy
\eqref{main1}--\eqref{main2}.
On the other hand, it is clear from the above construction that
if a function $h$ is orthogonal to the system $\{g_\lambda\}_{\lambda\in \Lambda_1}
\cup \{k_\lambda\}_{\lambda\in \Lambda_2}$, then
equations \eqref{main1}--\eqref{main2} are satisfied with $\{a_n\}=\{b_n\}$.
\smallskip

2. If $\{k_\lambda\}_{\lambda\in \Lambda}$ is an $M$-basis,
then the mixed system
$\{g_\lambda\}_{\lambda\in \Lambda_1}
\cup \{k_\lambda\}_{\lambda\in \Lambda_2}$ is always complete
in the case when $\Lambda_1$ or $\Lambda_2$ is a finite set. So we need to consider
only the partitions into infinite subsets. Furthermore, if the system
$\{g_\lambda\}_{\lambda\in \Lambda_1} \cup \{k_\lambda\}_{\lambda\in \Lambda_2}$
is incomplete, then, under some mild conditions on $\he$, there is a strong asymmetry between the two sets: $\Lambda_1$
should be a small (sparse) part of $\Lambda$ (see \cite[Theorems 1.2, 1.5]{bbb}).}
\end{remark}

The following 
remark shows that in Proposition \ref{firstred}
we can always {\it fix some spectral data} $(T, \mu)$
and consider only systems $\{k_\lambda\}_{\lambda\in \Lambda}$
with $\Lambda_1 \cap T= \emptyset$.

\begin{remark}
\label{imp0}
{\rm  Assume that $\{k_\lambda\}_{\lambda\in \Lambda}$ is an $M$-basis which
is not strong. 
Then, perturbing slightly the sequence $\Lambda_1$, we may construct a new sequence
$\tilde \Lambda_1$ with $\tilde \Lambda_1 \cap T
= \emptyset$
such that
$\{k_\lambda\}_{\lambda\in \tilde \Lambda_1\cup\Lambda_2}$
is an $M$-basis, but the system $\{g_\lambda\}_{\lambda\in \tilde \Lambda_1}
\cup \{k_\lambda\}_{\lambda\in \Lambda_2}$ is not complete in $\he$.

Indeed, let $\Lambda_1 = \{\lambda_j^1\}$ and let $\tilde \lambda_j^1 \notin T
$  be a perturbation of $\lambda_j^1$ so small that
$1/2 \le \prod_j \Big|\frac{t_n - \tilde \lambda_j^1}{t_n - \lambda_j^1}\Big| \le 2$ for any $n$.
It is easy to see that for sufficiently small perturbations
the equation \eqref{main1} is stable under
multiplication by $\prod_j \frac{z - \tilde \lambda_j^1}{z - \lambda_j^1}$:
$$
\frac{G_1(z)S_1(z)}{A(z)} \cdot \prod_j \frac{z -
\tilde \lambda_j^1}{z - \lambda_j^1} =
\sum_n
\frac{G(t_n)b_n}{\mu_n^{1/2}A'(t_n)(z-t_n)} \cdot \prod_j \frac{t_n - \tilde \lambda_j^1}{t_n - \lambda_j^1}.
$$
Thus, we have an equation of the form (\ref{main1}) with $\tilde G_1(z) =
G_1(z) \prod_j \frac{z - \tilde \lambda_j^1}{z - \lambda_j^1}$ in place of $G_1$
and with the same $\{b_n\}$ and $S_1$, whence
$\{g_\lambda\}_{\lambda\in \tilde \Lambda_1}
\cup \{k_\lambda\}_{\lambda\in \Lambda_2}$ is not complete in $\he$.
At the same time, it is easy to show that sufficiently small perturbations
preserve the property to be an $M$-basis. }
\end{remark}


\subsection{Outline of the proof}
\label{outline}

In Section \ref{suff} the sufficiency of (i) or (ii) for the strong $M$-basis
property is proved. As mentioned above, we need to show
that equations \eqref{main1}--\eqref{main2}
do not have a nontrivial solution $\{a_n\}=\{b_n\}$.
While case (i) follows essentially from comparing of the asymptotics
along the imaginary axis, in case (ii) a subtler argument
on the asymptotics of the zeros of $h$ and $S$ is used.

The proof of the converse statement is more involved.
It splits into four cases which will be treated separately:
\medskip

(I) $\inf_n \mu_n = 0$ and $\sum_n\mu_n=\infty$;
\smallskip

(II) $\inf_n\mu_n>0$, and there exists a subsequence
$n_k$ such that $d_{n_k} = o(|t_{n_k}|)$,
$k\to \infty$, and $d_{n_k} \gtrsim |t_{n_k}|^{-N}$ for some $N>0$;
\smallskip

(III) $\inf_n\mu_n>0$, and there exists a subsequence
$n_k$ such that $d_{n_k}= o(|t_{n_k}|^{-N})$ for any $N>0$
as $k\to \infty$;
\smallskip

(IV) $\inf_n\mu_n>0$, $\inf_n\frac{d_n}{|t_n|}>0$,
and there exists a subsequence $n_k$ such
that the converse to the estimate \eqref{mulacunary} holds:
$$
\mu_{n_k}=o\Bigl{(}\sum_{|t_l|\leq|t_{n_k}|}
\mu_l+t^2_{n_k}\sum_{|t_l|>|t_{n_k}|}\frac{\mu_l}{t^2_l}\Bigr{)},\qquad k\to\infty.
$$

The proof for the case (I) is given in Section \ref{case1}.
One of its ingredients is the following result from \cite[Theorem 1.1]{bb}:
{\it if $\sum_n \mu_n <\infty$, then the de Branges space with the spectral data
$(T, \mu)$ contains an exact system of reproducing kernel with incomplete
biorthogonal system}. Some extension of this result (with a simplified proof)
is given in Section \ref{infdim}.

Another ingredient is the following observation which says roughly that
the strong $M$-basis property for a de Branges space implies the same property for all
spaces constructed from a part of its spectral data.
This observation will be also useful in the proofs of other cases.

\begin{proposition}
\label{submu}
Let $\he$ be a de Branges space with the spectral data $(T, \mu)$
and $\sum_n \mu_n = \infty$.
Let $T^\circ \subset T$, 
let $\mu^\circ = \mu|_{T^\circ}$,
and let $\mathcal{H}(E^\circ)$, $E^\circ = A^\circ -iB^\circ$,
be the de Branges space constructed from the
spectral data $(T^\circ, \mu^\circ)$. If there exists an exact system of
reproducing kernels $\{k^\circ_\lambda\}_{\lambda\in \Lambda^\circ}$
in $\mathcal{H}(E^\circ)$ which is not
a strong $M$-basis and whose generating function $G^\circ$ satisfies
\begin{equation}
\label{rest}
|G^\circ(iy)|\gtrsim |y|^{-N}|A^\circ(iy)|, \qquad |y| \to\infty,
\end{equation}
for some $N>0$, then
there exists an $M$-basis of reproducing kernels in $\he$
which is not a strong $M$-basis.
\end{proposition}

It should be emphasized that in the statement of Proposition \ref{submu}
the system $\{k^\circ_\lambda\}_{\lambda\in \Lambda^\circ}$ can have incomplete
biorthogonal system (this version will be used in the proof of Case I)
or can be an $M$-basis but not a strong one.

The proofs for the cases (II)--(IV) consist of two steps. At the first step
(Section \ref{case24}), in each
of these cases we construct two real sequences $\{a_n\}$ and $\{b_n\}$
with the following \medskip properties:

(a) $\{a_n\} \in \ell^1$, $\{b_n\} \in \ell^1$, $a_n \ne 0$, $b_n\ne 0$;
\smallskip

(b) $\sum_n a_n b_n >0$;
\smallskip

(c) the entire functions $h$ and $S$ defined by
\begin{equation}
\label{dv}
\frac{h(z)}{A(z)} = \sum_n\frac{a_n\mu_n^{1/2}}{z-t_n},\qquad
\frac{S(z)}{A(z)} = \sum_n\frac{a_nb_n}{z-t_n}
\end{equation}
have infinitely many common real zeros $\{s_k\}$ such that
$\dist (s_k, T) \gtrsim |s_k|^{-N}$, for some $N>0$.
\smallskip

(d) the function $h$ satisfies $|h(iy)| \gtrsim |y|^{-1}|A(iy)|$
(note that this estimate is trivially satisfied if $a_n>0$ for every $n$).
\medskip

Existence of common zeros is the crucial part of the proof. Once
such sequences are constructed, we obtain an $M$-basis of
reproducing kernels in $\he$ which is not strong by a certain
perturbation argument. This method was suggested in \cite{bbb}.
The following proposition (whose proof is given in
Section \ref{proofcod}) completes the proof of Theorem
\ref{mainn}.

\begin{proposition}
\label{cod}
Let $\he$ be a de Branges space with the spectral data $(T, \mu)$
such that $\sum_n \mu_n = \infty$,
and assume that there exist two sequences $\{a_n\}$ and $\{b_n\}$
satisfying $(a)$--$(d)$. Then there exists an $M$-basis of reproducing
kernels $\{k_\lambda\}$ in $\he$ which is not a strong $M$-basis and
such that its generating function $G$ satisfies $|G(iy)|\gtrsim |y|^{-N} |A(iy)|$
for some $N>0$.
\end{proposition}


\section{Sufficiency of $(\rm i)$ and $(\rm{ii})$ in Theorem \ref{mainn}}
\label{suff}

In this section we will show that each of the conditions (i) and (ii)
implies that any $M$-basis in the corresponding de Branges space is a strong $M$-basis.
The proofs of these two cases will be essentially different.

Assume that $\he$ does not have strong $M$-basis property.
Then, by Proposition \ref{firstred}, there exist
a partition $\Lambda = \Lambda_1 \cup \Lambda_2$
and a nonzero sequence $\{a_n\} \in \ell^2$ such that for
some entire functions $S_1$ and $S_2$ equations
\eqref{main1}--\eqref{main2} hold with $b_n = a_n$.

Consider the product of the equations \eqref{main1}--\eqref{main2},
\begin{equation}
\label{product}
  A^2(z)\biggl{(}\sum_{n} \frac{\overline{a_n}\mu^{1/2}_n}{z-t_n}\biggr{)}
  \biggr{(}\sum_{n}\frac{a_nG(t_n)}{\mu^{1/2}_nA'(t_n)(z-t_n)}\biggr{)}=G(z)S(z)
\end{equation}
By \eqref{st12}, we have $S(t_n)=A'(t_n)|a_n|^2$ and hence 
\begin{equation}
\label{product1}
  \frac{S(z)}{A(z)}=R(z)+\sum_n\frac{|a_n|^2}{z-t_n},
\end{equation}
where $R$ is an entire function.


\subsection{(i)$\Longrightarrow$ Strong $M$-basis Property. \label{sub1} }
The series
$$
  \sum_n\frac{G(t_n)}{A'(t_n)(z-t_n)}
$$
converges since $\Big\{ \frac{G(t_n)}{\mu^{1/2}_n A'(t_n)t_n}\Big\}
\in\ell^2 $ and $\{\mu_n^{1/2}\} \in \ell^2$.
Let $\lambda_0$ be an arbitrary zero of $G$. From inclusion
$\frac{G(z)}{z-\lambda_0}\in\he$ we deduce that
\begin{equation}
\label{c1}
\frac{G(z)}{A(z)}=c+\sum_n\frac{G(t_n)}{A'(t_n)(z-t_n)}.
\end{equation}
Then we may rewrite \eqref{product} as
\begin{equation}
  \label{Q}
  \sum_n\frac{a_nG(t_n)}{\mu^{1/2}_nA'(t_n)(z-t_n)}
  \cdot \sum_n\frac{\overline{a_n}\mu^{1/2}_n}{z-t_n}=
  \biggl{(}c+\sum_n\frac{G(t_n)}{A'(t_n)(z-t_n)}\biggr{)}
  \cdot\biggl{(}R(z)+\sum_n\frac{|a_n|^2}{z-t_n}\biggr{)}.
\end{equation}

Assume first that $c = 0$ in \eqref{c1}. Let us show that in this case
the system $\{g_\lambda\}_{\lambda\in\Lambda}$ is orthogonal to a
function $B_0\in\he$,
$$
B_0=\sum_n\mu^{1/2}_n(-1)^n\cdot\frac{k_{t_n}}{\|k_{t_n}\|}.
$$
Indeed,
$$
  \biggl{(}\frac{G(z)}{z-\lambda},B_0\biggr{)}=\pi\sum_n\frac{G(t_n)}{A'(t_n)(t_n-\lambda)}=0,\qquad \lambda\in\Lambda,
$$
and hence the system $\{g_\lambda\}$ is not complete. Thus $c\neq0$.

It is a standard fact that if
$\sum_n\frac{|c_n|}{1+|t_n|}<\infty$, then the Cauchy transform
$\sum_n\frac{c_n}{z-t_n}$ is a function of Smirnov class (that is,
a ratio $g/h$ of two bounded analytic functions, where $h$ is
outer) both in the upper half-plane $\mathbb{C}_+$ and in the
lower half-plane $\mathbb{C}_-$ (see, e.g., \cite[Part II, Chapter
1, Section 5]{hj}). We conclude that $S/A$ and, hence, $R$ is in
the Smirnov class both in $\mathbb{C}_+$ and $\mathbb{C}_-$. By
M.G. Krein's theorem (see, e.g., \cite[Part II, Chapter 1]{hj})
$R$ is of zero exponential type. The left-hand side of \eqref{Q}
tends to zero along the imaginary axis. So, $R\equiv 0$. Now we
again use the fact that $\sum_n\mu_n<\infty$ to see that
$$
   \sum_n\frac{\overline{a_n}\mu^{1/2}_n}{iy-t_n}=O(|y|^{-1}), \qquad |y|\to \infty.
$$
By \eqref{product1} and \eqref{c1} we have $1\lesssim|(G/ A)(iy)|$, $|y|^{-1} \lesssim |(S/ A)(iy)|$.
On the other hand,
$$
\sum_n\frac{a_nG(t_n)}{\mu^{1/2}_nA'(t_n)(iy-t_n)}=o(1),
\qquad |y| \to\infty,
$$
that contradicts to equation \eqref{Q},
since the right-hand side is $\asymp |y|^{-1}$, while the left-hand side
is $o(|y|^{-1})$, $|y| \to\infty$.


\subsection{(ii)$\Longrightarrow$ Strong $M$-basis Property.}
The proof consists of three steps.
\smallskip

{\it Step 1.} As in the previous proof, we prove first
(using a different argument) that $R \equiv0$ in \eqref{product1}.
Let $2\varphi$ be a smooth increasing
branch of the argument of $\Theta_E$ on $\rl$ (the phase function
for $E$, see Subsection \ref{debr0}).
By \eqref{use}, 
the derivative
$\varphi'$ is bounded on $\rl$. 

By \cite[Lemma 5.1]{bbb}, $R$ is at most a polynomial.
If $R$ is not identically zero, then it follows from
\eqref{product1} that the zeros $s_n$ of $S$ satisfy
$\dist (s_n,T) \to 0$, $|s_n| \to \infty$. Note, however,
that the zeros of $S_2$ do not depend on the choice
of the spectral data. So for any other spectral data $(\tilde T, \tilde \mu)$
we have $\dist (s_n,\tilde T) \to 0$, $|s_n| \to \infty$.
However, since $\varphi'$ is bounded on $\rl$, we have
$\dist (T, \tilde T) \gtrsim 0$ (recall that $T = \{\Theta = -1\}$
and $\tilde T = \{\Theta = \alpha\}$ for some $\alpha\in \tz$, $\alpha \ne -1$).
We conclude that $R\equiv0$.
\smallskip

{\it Step 2.} The function $S$ vanishes only at the zeros of $\sum |a_n|^2/(\cdot - t_n)$, and hence has only real zeros.
Note also that $|S(iy)| \gtrsim |y|^{-1}|A(iy)|$.
If $S_2$ has only finite number of real zeros, then
$|S_1(iy)| \gtrsim |y|^{-N}|A(iy)|$ for some $N>0$.
This means that $G_1$ is a polynomial (the case excluded by Remark \ref{st14}).

Thus, 
the Cauchy transforms $\sum_n\frac{\overline{a_n}\mu^{1/2}_n}{z-t_n}$
and $\sum_n\frac{|a_n|^2}{z-t_n}$
have common zeros $s_n\in (t_{n-1}, t_n)$
for $n$ in some {\it infinite} set  $\mathcal{N}$.
We will assume that $s_n\to \infty$. 
\smallskip

{\it Step 3.} We have
$$
  \sum_{t_k<t_n}\frac{|a_k|^2}{s_n-t_k}+\frac{|a_n|^2}{s_n-t_n}+
  \sum_{t_k>t_n}\frac{|a_k|^2}{s_n-t_k}=0.
$$
The first sum is greater than $\frac{C}{t_n}$ for some $C>0$.
The last sum is $o(t_n^{-1})$. So,
$\frac{|a_n|^2}{\Delta_n}\geq \frac{C}{t_n}$, where $\Delta_n=t_n-s_n$,
i.e., $\Delta_n\leq \frac{|a_n|^2t_n}{C}$.
On the other hand,
$$
 \sum_{k\neq n}\frac{\overline{a_k}\mu^{1/2}_k}{s_n-t_k}+
 \frac{\overline{a_n}\mu^{1/2}_n}{s_n-t_n}=0.
$$
Furthermore,
$$
  \biggl{|}\sum_{|t_k|\leq t_n, k\neq n}\frac{\overline{a_k}
  \mu^{1/2}_k}{s_n-t_k}\biggr{|}\lesssim\frac{1}{t_n}
  \biggl{(}\sum_{|t_k|\leq t_n}\mu_k\biggr{)}^{1/2}\cdot
  \|\{a_n\}\|_{\ell^2} \lesssim \frac{\mu^{1/2}_n}{t_n},
$$
and
$$
  \biggl{|}\sum_{|t_k|>t_n}\frac{\overline{a_k}\mu^{1/2}_k}
  {s_n-t_k}\biggr{|}\lesssim \sum_{|t_k|>t_n}\frac{\mu^{1/2}_k}{|t_k|}
  \cdot|a_k|\leq\biggl{(}\sum_{|t_k|>t_n}
  \frac{\mu_k}{t^2_k}\biggr{)}^{1/2}\cdot \|\{a_n\}\|_{\ell^2}
  \lesssim \frac{\mu_n^{1/2}}{t_n}.
$$
We conclude that $\frac{|a_n|\mu^{1/2}_n}{\Delta_n}\lesssim\frac{\mu^{1/2}_n}{t_n}$.
On the other hand,
$\frac{\mu^{1/2}_n}{t_n|a_n|}\lesssim\frac{|a_n|\mu^{1/2}_n}{\Delta_n}$
and so $\frac{1}{|a_n|}\lesssim 1$ for infinitely many indices $n$. We come to
a contradiction.
\bigskip


\section{Examples of $M$-bases which are not strong:  case I.}
\label{case1}

\subsection{Proof of Proposition \ref{submu}.}
Let $\{k^\circ_\lambda\}_{\lambda\in \Lambda^\circ}$
be an exact system of reproducing kernels in $\mathcal{H}(E^\circ)$ which is not
a strong $M$-basis and let $G^\circ$ be its generating function.
By Proposition \ref{firstred} and Remark \ref{imp0}
there exists a partition $\Lambda^\circ = \Lambda_1^\circ \cup \Lambda_2^\circ$
such that $\Lambda_1^\circ \cap T^\circ = \emptyset$ and
\begin{equation}
\label{bor}
\begin{aligned}
\frac{G_1^\circ(z)S^\circ_1(z)}{A^\circ(z)} & =
\sum_{t_n\in T^\circ} \frac{b^\circ_n G^\circ(t_n)}{\mu^{1/2}_n (A^\circ)'(t_n)(z-t_n)},
\\
\frac{G_2^\circ(z)S^\circ_2(z)}{A^\circ(z)}  & =
\sum_{t_n\in T^\circ}\frac{\overline{a_n^\circ}\mu^{1/2}_n}{z-t_n}
\end{aligned}
\end{equation}
for some sequences $\{a^\circ_n\}$, $\{b^\circ_n\} \in \ell^2$ and some
entire functions $S^\circ_1$ and $S^\circ_2$. Here we do not exclude the case
when the system $\{g_\lambda^\circ\}_{\lambda\in \Lambda^\circ}$
is incomplete in $\mathcal{H}(E^\circ)$. In this case $G_1^\circ = G$
and $G_2^\circ \equiv 1$.

Let $\he$ be a de Branges space with the spectral data $(T, \mu)$. Then $E=A-iB$,
where $\mathcal{Z}_A = T$, and we can write $A=A^\circ \tilde A$
with $\mathcal{Z}_{\tilde A} = \tilde T =
T \setminus T^\circ$. Define $a_n = a_n^\circ$,
$t_n\in T^\circ$, and $a_n = 0$, $t_n\in \tilde T$, and define $b_n$ analogously.
Then, multiplying equations \eqref{bor} by $\tilde A$ and using the fact that
$A'(t_n) = (A^\circ)'(t_n)\tilde A(t_n)$, $t_n \in T^\circ$, we get
\begin{equation}
\label{bor1}
\begin{aligned}
\frac{G_1^\circ(z)S^\circ_1(z)\tilde A(z)}{A(z)} & =
\sum_{t_n \in T} \frac{b_n G^\circ(t_n)\tilde A(t_n)}
{\mu^{1/2}_n A'(t_n)(z-t_n)},
\\
\frac{G_2^\circ(z)\tilde A(z)S^\circ_2(z)}{A(z)} &  =
\sum_{t_n \in T} \frac{\overline{a_n}\mu^{1/2}_n}{z-t_n}.
\end{aligned}
\end{equation}
Put $G_1 = G_1^\circ$, $G_2 =  G_2^\circ \tilde A$, $G=G_1 G_2$, and $\Lambda =
\Lambda^\circ \cup\tilde T$. Then equations \eqref{bor1}
are of the form \eqref{main1}--\eqref{main2}
for the partition $\Lambda_1 = \Lambda_1^\circ$
and $\Lambda_2 = \Lambda_2^\circ \cup \tilde T$ of $\Lambda$.
Hence the system
$\{g_\lambda\}_{\lambda\in \Lambda_1} \cup \{k_\lambda\}_{\lambda\in \Lambda_2}$
is incomplete.

It remains to show that the system $\{k_\lambda\}_{\lambda\in \Lambda}$
is an $M$-basis in $\he$.
If $\{k_\lambda\}_{\lambda\in \Lambda}$ is incomplete,
then there exists an entire function $H$ such that
$GH = G^\circ \tilde A H \in \he$. Hence, we have
$$
G^\circ(z) \tilde A(z) H(z) = A(z) \sum_n \frac{c_n\mu_n^{1/2}}{z-t_n},
$$
where $\{c_n\} \in \ell^2$ and $c_n=0$ when $t_n\in \tilde T$.
Dividing by $\tilde A$ we have $G^\circ(z) H(z) = A^\circ(z) \sum_{t_n\in T^\circ}
\frac{c_n\mu_n^{1/2}}{z-t_n}$, whence $G^\circ H \in \mathcal{H}(E^\circ)$,
a contradiction with completeness of $\{k^\circ_\lambda\}_{\lambda\in \Lambda^\circ}$.
It is also clear from the above that $G/(\cdot-\lambda) \in \he$ for any
$\lambda\in \Lambda$, so $G$ is the generating function
of an exact system of reproducing kernels in $\he$.

Now assume that the biorthogonal system
$\{g_\lambda\}_{\lambda\in \Lambda}$  is incomplete in $\he$. Let
$(U, \nu)$, $U=\{u_n\}$,
be the spectral data for $\he$ corresponding to the function
$E_\alpha = (\alpha E-E^*)/2$ for some $\alpha\in \tz$, $\alpha \ne -1$.
By the remarks in Subsection \ref{clar}, since $A\notin\he$
and $\mu(\rl) = \infty$, there is no exceptional value $\alpha$
and, thus, $\nu(\rl) = \infty$.

By the arguments from Subsection \ref{reduc} (or from \cite[Section 2]{bb}),
there exists a sequence $\{c_n\} \in \ell^2$ and a nonzero entire function $V$
such that
\begin{equation}
\label{bor2}
\frac{G(z)V(z)}{E_\alpha(z)} = \sum_{n}\frac{c_n G(u_n)}
{\nu^{1/2}_n E_\alpha'(u_n)(z-u_n)}.
\end{equation}
Comparing the residues, we see that $V(u_n) = \nu_n^{-1/2} c_n$ and so
$V\in L^2(\nu)$.
The function $E_\alpha/A$ is in the Smirnov class in $\cp$ and $\mathbb{C}_-$
and $|A(iy)/E_\alpha(iy)| \gtrsim |y|^{-1}$
(note that $\frac{E_\alpha}{A} = \frac{\Theta_E-\alpha}{1+\Theta_E}$).
Since $|G^\circ(iy)|\gtrsim |y|^{-N}|A^\circ(iy)|$, we have
$$
|G(iy)| = |G^\circ(iy)\tilde A(iy)| \gtrsim |y|^{-N} |A(iy)|\gtrsim
|y|^{-N-1} |E_\alpha(iy)|.
$$
By Krein's theorem we conclude
that $V$ is a polynomial. Since $V\in L^2(\nu)$
and $\nu(\rl) = \infty$, we have $V\equiv 0$.
\qed

\begin{remark}
\label{imp1}
{\rm It follows from the proof of Proposition \ref{submu}
that the dimension of the orthogonal complement to the system
$\{g^\circ_\lambda\}_{\lambda\in \Lambda_1^\circ} \cup
\{k^\circ_\lambda\}_{\lambda\in \Lambda_2^\circ}$
in $\mathcal{H}(E^\circ)$
coincides with the dimension of the orthogonal complement to
$\{g_\lambda\}_{\lambda\in \Lambda_1} \cup \{k_\lambda\}_{\lambda\in \Lambda_2}$
in $\he$. Indeed, there is a one-to-one correspondence between
sequences $a_n^\circ = b_n^\circ$ satisfying \eqref{bor}
and sequences $a_n = b_n$ satisfying \eqref{bor1}. }
\end{remark}

\begin{remark}
\label{imp2}
{\rm Note that the condition $|G^\circ(iy)|\gtrsim |y|^{-N}|A^\circ(iy)|$
was used only to prove the completeness of the system biorthogonal
to $\{k_\lambda\}_{\lambda\in \Lambda}$. Both equations
\eqref{main1}--\eqref{main2} for some partition of $\Lambda$
and the completeness of
$\{k_\lambda\}_{\lambda\in \Lambda}$ follow without this assumption. }
\end{remark}


\subsection{(I)$\Longrightarrow$ $\he$ does not have Strong $M$-Basis Property.}
Since $\liminf_{|n|\to\infty}
\mu_n = 0$, we may choose a subsequence $T^\circ = \{t_{n_k}\}$ such that
$\sum_k \mu_{n_k} <\infty$, and let $\mathcal{H}(E^\circ)$
be the de Branges space with the spectral data $(T^\circ, \mu^\circ)$,
$\mu^\circ = \mu|_{T^\circ}$ (thus $\mu^\circ(\rl) < \infty$).

By Theorem~1.1 in \cite{bb} and its proof, there exists an exact system of
reproducing kernels $\{k^\circ_\lambda\}_{\lambda\in \Lambda^\circ}$
in $\mathcal{H}(E^\circ)$
whose generating function $G^\circ$ satisfies
$|G^\circ(iy)|\gtrsim |y|^{-1}|A^\circ(iy)|$ 
and such that its biorthogonal system
$\{g^\circ_\lambda\}_{\lambda\in \Lambda^\circ}$
is incomplete in $\mathcal{H}(E^\circ)$.
For a simplified proof of this (and slightly more general)
statement see Proposition \ref{bior}. 
Then by Proposition \ref{submu} there exists an $M$-basis of
reproducing kernels in $\he$ which is not a strong $M$-basis.
\qed
\bigskip


\section{Cases II, III and IV: construction of common zeros.}
\label{case24}

As it was explained in Subsection \ref{reduc}
we need to construct sequences $\{a_n\}$, $\{b_n\}\in\ell^2$
(or, which is the same, two vectors $h$, $\tilde h$)
and a generating function $G$ of an $M$-basis of reproducing kernels
such that equations \eqref{main1}--\eqref{main2} hold
and $\sum_na_n\overline{b_n}=(h_1,h_2)\neq 0$.

By Proposition \ref{submu}, it suffices to construct
such example for any restriction $\mu^\circ = \mu|_{T^\circ}$, $T^\circ\subset T$, 
with additional restriction \eqref{rest}.

In this section we will make the first step and construct, in each of the cases
(II)--(IV) sequences $\{a_n\}$, $\{b_n\}$
satisfying conditions (a)--(d) from  Subsection \ref{outline} with $T$ replaced by some $\mathcal{N}\subset T$.
In each of the cases, the existence of common zeros
is proved by a certain fixed point argument.


\subsection{Case (II)}

Without loss of generality, we assume that $T$ is a positive sequence.
Passing to a subsequence, assume that $|t_{n_k}|^{-N}\lesssim d_{n_k}
\lesssim |t_{n_k}|k^{-6}$ and $t_{n_{k+1}}>2 t_{n_k}$.

Put $\mathcal{N}=\{n_k\} \cup \{n_k+1\}$.
Let $s_k \in(t_{n_k}, t_{n_{k+1}})$ be such that
$$
\frac{\mu^{1/2}_{n_k}}{\Delta^l_k}=\frac{\mu^{1/2}_{n_{k}+1}}{\Delta^r_k},
\qquad \Delta^l_k=s_k - t_{n_k},\qquad \Delta^r_k=t_{n_{k}+1}-s_k.
$$
We want to choose the coefficients $(a_n)_{n\in \mathcal N}, (b_n)_{n\in \mathcal N}$ in such a way that the Cauchy transforms
$$
\sum_{n\in \mathcal N}\frac{a_n\mu_n^{1/2}}{z-t_n},\qquad
\sum_{n\in \mathcal N}\frac{a_nb_n}{z-t_n}
$$
vanish at the points $s_k$. We assume that
$\mu_{n_k+1}\geq \mu_{n_k}$ (otherwise, the construction should be
modified in an obvious way). Put
$$
  a_{n_k}=\frac{r_k}{k^2},\qquad a_{n_k+1}=\frac{1}{k^2},
$$
$$
  b_{n_k}=\frac{\mu^{1/2}_{n_k}}{\mu^{1/2}_{n_k+1}}\cdot\frac{1}{k^2}=
  \frac{\Delta^l_k}{\Delta^r_kk^2},
  \qquad b_{n_k+1}=\frac{q_k}{k^2}.
$$
The numbers $r_k, q_k\in(1/2,3/2)$ will be our free
parameters. We have a system of equations
$$
\begin{cases}
\frac{r_k\mu^{1/2}_{n_k}}{k^2\Delta^l_k}-\frac{\mu^{1/2}_{n_k+1}}{k^2\Delta^r_k}
+ \sum\limits_{n\neq n_k, n_k+1}\frac{a_n\mu^{1/2}_n}{s_k-t_n} = 0 \\
\frac{r_k}{k^2} \cdot\frac{\Delta^l_k}{\Delta^r_kk^2}\cdot
\frac{1}{\Delta^l_k}-\frac{q_k}{k^2}\cdot\frac{1}{k^2\Delta^r_k}
+ \sum\limits_{n\neq n_k, n_k+1}\frac{a_nb_n}{s_k-t_n} = 0.
\end{cases}
$$

Let $\varepsilon>0$ be a small number to be fixed later on.
Since $\frac{t_{n_k}}{k^5d_{n_k}}\rightarrow\infty$, we can
start with the sequence $\{t_{n_k}\}$ so sparse that
$$
\sum_{l\neq k}\frac{\mu^{1/2}_{n_l}+\mu^{1/2}_{n_l+1}}{|s_k-t_{n_l}|} \lesssim \frac{\varepsilon}{k^5d_{n_k}}.
$$
Furthermore, we can rewrite our system as
$$
\begin{cases}
r_k
+ \frac{\Delta^r_kk^2}{\mu^{1/2}_{n_k+1}} \cdot
\sum\limits_{j\ne k}\bigg(\frac{r_j \mu_{n_j}^{1/2}}{j^2(s_k - t_{n_j})} +
\frac{\mu_{n_j+1}^{1/2}}{j^2(s_k - t_{n_j+1})} \bigg)
=1,
\\
r_k-q_k
+ \Delta^r_kk^4 \sum\limits_{j\ne k}\bigg(\frac{r_j \mu_{n_j}^{1/2}}{j^4\mu_{n_j+1}^{1/2}
(s_k - t_{n_j})} +
\frac{q_j}{j^4 (s_k - t_{n_j+1})} \bigg)
=0.
\end{cases}
$$
The unperturbed system
$$
\begin{cases}
r_k=1,
\\
r_k-q_k=0
\end{cases}
$$
has a unique solution $r_k=q_k=1$. On the other hand,
$$
\frac{\Delta^r_kk^2}{\mu^{1/2}_{n_k+1}}
\sum\limits_{j\ne k}\Bigl(\frac{\mu_{n_j}^{1/2}}{j^2|s_k - t_{n_j}|}+\frac{\mu_{n_j+1}^{1/2}}{j^2|s_k - t_{n_j+1}|}\Bigr)
\lesssim \varepsilon
$$
and
$$
\Delta^r_kk^4 \sum\limits_{j\ne k}\frac1{j^4}\Bigl(\frac{\mu_{n_j}^{1/2}}{\mu_{n_j+1}^{1/2}|s_k - t_{n_j}|} +
\frac{1}{|s_k - t_{n_j+1}|} \Bigr)\lesssim \varepsilon
$$
uniformly with respect to $r_j,q_j\in (1/2,3/2)$.
This means that for sufficiently small $\varepsilon$, our system has a unique solution $r_k,q_k\in (1/2,3/2)$.

Since $a_n$ and $b_n$ are positive, conditions (a), (b) and (d) are clearly satisfied.
It remains to notice that $1\lesssim \mu_n \lesssim t_n^2$
(recall that $\sum \mu_n t_n^{-2} <\infty$). Since $d_{n_k} \ge c|t_{n_k}|^{-N}$
for some $c, \, N>0$, we have $s_k-t_{n_k} \gtrsim s_{k}^{-N-1}$,
the same is true for $t_{n_k+1} - s_k$, whence $\dist(s_k, T) \gtrsim |s_k|^{-N-1}$.
\qed


\subsection{Case (III)}
The previous proof works as soon as we have a subsequence
of intervals $d_{n_k}$ with $d_{n_k} = o(|t_{n_k}|)$
and $d_{n_k} \gtrsim |t_{n_k}|^{-N}$
for some $N>0$. The problem is with the case where all
short intervals are in fact extremely (super-polynomially) short. In this case
we cannot put the common zero in this interval if we want to have a separation condition
$\dist(s_k, T) \gtrsim |s_k|^{-N}$. Therefore, we will choose the common zero
outside this interval and so $a_n$ and $b_n$ cannot be taken positive for all $n$.

By Proposition \ref{submu} we can rarify our sequence $T$. So, without loss of generality,
assume that there exists a sequence of indices $\{n_k\}$
such that $\mu_{n_k}\gtrsim1$, $d_{n_k} \le |t_{n_k}|^{-2}$ and $t_{n_{k+1}} > 2t_{n_k +1}$. Moreover, let $t_{n_k}$ grow super-exponentially,
namely, let $k^6 t_{n_j}< \varepsilon t_{n_k}$ when $j<k$ for some $\varepsilon>0$.
Assume additionally that $\mu_{n_k}\leq\mu_{n_k+1}$.
Then we construct common zeros $s_k$ at the points $t_{n_k}-1$
(if $\mu_{n_k} > \mu_{n_k+1}$, then the points $t_{n_k+1}+1$
should be taken as common zeros).  Put
\begin{equation}
\begin{cases}
a_{n_k}=\frac{r_k}{k^2},\quad a_{n_k+1}=
-\frac{1}{k^2}\cdot\frac{\mu^{1/2}_{n_k}}{\mu^{1/2}_{n_k+1}}(d_{n_k}+1),
\\
b_{n_k}=\frac{\mu^{1/2}_{n_k}}{\mu^{1/2}_{n_k+1}}\cdot\frac{1}{k^2},
\quad b_{n_k+1}=\frac{q_k}{k^2}.
\end{cases}
\label{choose}
\end{equation}
We choose $a_0=b_0=10$. 
Set $\mathcal N = \{0\}\cup \{n_k\}\cup \{n_k+1\}$.
We have a system of equations
$$\sum_{n\in \mathcal N}\frac{a_n\mu^{1/2}_n}{s_k-t_n}=\sum_{n\in \mathcal N}\frac{a_nb_n}{s_k-t_n}=0,$$
and try to find a solution $(r_k, q_k)$ of this system such that $r_k, q_k\in(1/2, 3/2)$.
Rewrite our system using \eqref{choose}:
$$
\begin{cases}
\sum\limits_{l<n_k}\frac{a_l\mu^{1/2}_l}{s_k-t_l}-\frac{r_k}{k^2}\mu^{1/2}_{n_k}+\frac{1}{k^2}\mu^{1/2}_{n_k}+
\sum\limits_{l>n_k+1}\frac{a_l\mu^{1/2}_l}{s_k-t_l}=0\\
\sum\limits_{l<n_k}\frac{a_lb_l}{s_k-t_l}-\frac{r_k}{k^4}\cdot\frac{\mu^{1/2}_{n_k}}{\mu^{1/2}_{n_k+1}}+\frac{q_k}{k^4}\cdot\frac{\mu^{1/2}_{n_k}}{\mu^{1/2}_{n_k+1}}+
\sum\limits_{l>n_k+1}\frac{a_lb_l}{s_k-t_l}=0,
\end{cases}
$$
or, equivalently,
\begin{equation}
\begin{cases}
 r_k=1+\frac{k^2}{\mu^{1/2}_{n_k}} \cdot \Bigl(
\sum\limits_{j\neq k}
\frac{r_j}{j^2}\frac{\mu^{1/2}_{n_j}}{s_k-t_{n_j}} -
\sum\limits_{j\neq k}\frac{1}{j^2}
\frac{\mu^{1/2}_{n_j}(d_{n_j}+1)}{s_k-t_{n_j+1}}+\frac{a_0\mu^{1/2}_0}{s_k-t_0}\Bigr) 
\\
%
%
r_k-q_k=\frac{k^4\mu^{1/2}_{n_k+1}}{\mu^{1/2}_{n_k}} \cdot
\Bigl(
\sum\limits_{j\neq k}\frac{r_j}{j^4}
\frac{\mu^{1/2}_{n_j}}{\mu^{1/2}_{n_j+1}(s_k-t_{n_j})}-
\sum\limits_{j\neq k}\frac{q_j}{j^4}\frac{\mu^{1/2}_{n_j}(d_{n_j}+1)}
{\mu^{1/2}_{n_j+1}(s_k-t_{n_j +1})}+\frac{a_0b_0}{s_k-t_0}\Bigr) 
\label{rsystem}
\end{cases}
\end{equation}
where all the sums should be thought of as small perturbations
of the block linear system $r_k=1$, $r_k-q_k =0$. 
The terms with $l=0$ 
can be made arbitrary small by passing to sufficiently sparse $\{n_k\}$ 
since $\mu^{1/2}_n=o(t_n)$.
In addition we can assume that $\frac{k^4\mu^{1/2}_{n_k+1}}{t_{n_k}}<2^{-k}\varepsilon \min_{j<k}\mu^{-1/2}_{n_j+1}$.

Next we want to estimate coefficients at $r_j$ in
the first equation in \eqref{rsystem}.  If $j<k$ then
$$
\frac{k^2}{\mu^{1/2}_{n_k}}\cdot\frac{\mu^{1/2}_{n_j}}{|s_k-t_{n_j}|}
\lesssim\frac{k^2\mu^{1/2}_{n_j}}{s_k}\lesssim\frac{k^2t_{n_j}}{t_{n_k}}\lesssim \frac{\varepsilon}{k^4}.
$$
If $j>k$ then
$$
\frac{k^2}{\mu^{1/2}_{n_k}}\cdot\frac{\mu^{1/2}_{n_j}}{|s_k-t_{n_j}|}
\lesssim\frac{k^2\mu^{1/2}_{n_j}}{t_{n_j}}<\frac{j^2\mu^{1/2}_{n_j}}{t_{n_j}}
\lesssim\varepsilon 2^{-j}.
$$
Analogous estimates hold for coefficients at $r_j$ and $q_j$
in the second equation in \eqref{rsystem}
and  for the sum $\frac{k^2}{\mu^{1/2}_{n_k}}\cdot
\sum_{j\neq k}\frac{1}{j^2}\frac{\mu^{1/2}_{n_j}(d_{n_j}+1)}{s_k-t_{n_j+1}}$.
Thus, if we define $A_{kj}$ as the coefficient at $r_j$
in the first equation in \eqref{rsystem}
and  $B_{kj}$ and $C_{kj}$ as the coefficients at $r_j$ and $q_j$
in the second equation in \eqref{rsystem},
then we have $\sum_{k,j: k\ne j} |A_{kj}|+ |B_{kj}|+ |C_{kj}| \lesssim\varepsilon$.
The unperturbed system has the solution $r_k=q_k=1$.
Hence, for sufficiently small $\varepsilon$ we can find the solutions of the perturbed system
$r_k, q_k \in (1/2,3/2)$.
As a consequence, 
$\sum_{n\in\mathcal N} a_nb_n>0$.

It remains to show that the function $h$ satisfies property (d).
This is a nontrivial matter since now $a_n$ change the sign.
\medskip

{\it Proof of} (d). 
By the construction of $a_{n_k}$
and $a_{n_k+1}$ we have
$$
D_k := a_{n_k}\mu_{n_k}^{1/2} + a_{n_k+1}\mu_{n_k+1}^{1/2}
= \frac{\mu_{n_k}^{1/2}}{k^2} (r_k-1-d_{n_k})
=
{\sum}_{l\neq n_k, n_k+1}\frac{a_l\mu^{1/2}_l}{s_k-t_l}
- \frac{\mu_{n_k}^{1/2}}{k^2}  d_{n_k}.
$$
Choosing the sequence $t_{n_k}$ sufficiently sparse
we can make the value of $D_k$ 
as small as we wish
(uniformly for all values of $r_k$ and $q_k \in (1/2, 3/2)$.
Thus, the series $\sum_k D_k$ converges absolutely
and its sum can be made as small as we wish.
Thus, we can assume that 
\begin{equation}
\label{ddd}
a_0\mu_0^{1/2} + \sum_k D_k 
>0.
\end{equation}

We have
$$
\begin{aligned}
\frac{h(iy)}{A(iy)} & = \frac{a_0\mu_0^{1/2}}{iy-t_0}
+ \sum_{k}\Bigl(\frac{a_{n_k}\mu_{n_k}^{1/2}}{iy-t_{n_k}} +
\frac{a_{n_k+1}\mu_{n_k+1}^{1/2}}{iy-t_{n_k+1}}  \Bigr) 
\\
& =
\frac{a_0\mu_0^{1/2}}{iy-t_0}
+ \sum_{k} \frac{D_k}{iy-t_{n_k}} +
\sum_{k}
\frac{a_{n_k+1}\mu_{n_k+1}^{1/2}(t_{n_k+1} -t_{n_k})}{(iy-t_{n_k+1})(iy-t_{n_k+1})}.
\end{aligned}
$$
Since  $d_{n_k} = t_{n_k+1} -t_{n_k} \le  |t_{n_k}|^{-2}$ 
and $\mu_{n_k+1} \lesssim t_{n_k}^2$, we conclude that
$$
\sum_{k}
\frac{a_{n_k+1}\mu_{n_k+1}^{1/2}(t_{n_k+1} -t_{n_k})}{(iy-t_{n_k+1})(iy-t_{n_k+1})} =
O\Big(\frac{1}{y^2}\Big), \qquad |y| \to \infty,
$$
and so, by \eqref{ddd}, $|h(iy)|\gtrsim |y|^{-1}|A(iy)|$.
\qed

\subsection{Case (IV)}
Without loss of generality we assume that $T$ is a positive lacunary sequence.
We know that \eqref{mulacunary} does not hold.
Consider the case when there exists a subsequence
of indices $\{n_k\}$ such that $\sum_{l<n_k}\mu_l\gtrsim \mu_{n_k}$.
We choose rapidly increasing sequences of indices $\{n_k\}$ and $\{m_k\}$ such that
\begin{equation}
\label{nkmk}
n_{k-1}<m_k<n_k,\qquad \sum_{l=m_k}^{n_k-1}\mu_l>k^6\mu_{n_k}.
\end{equation}
Put
$$
a_l=\mu^{1/2}_l\cdot\biggl{(}\sum_{p=m_k}^{n_k-1}\mu_p\biggr{)}^{-{1/2}}\cdot\frac{1}{k},\quad m_k\leq l\leq n_k-1,
\qquad a_{n_k}=\frac{1}{k^2}.
$$
Set $\mathcal N=\cup_k[m_k,n_k]$. 
We can start with $\{n_k\}$, $\{m_k\}$
increasing  so fast 
that
\begin{equation}
\sum_{n<m_k,\, n\in\mathcal N}a_n\mu^{1/2}_n=o\Bigl(\frac{1}{k}\cdot
\Bigl(\sum_{l=m_k}^{n_k-1}\mu_l\Bigr)^{1/2}\Bigr),\quad
\sum_{n>n_k,\, n\in\mathcal N}\frac{a_n\mu^{1/2}_n}{t_n}=o\Bigl( \frac{1}{kt_{n_k}}\Bigr),\qquad k\to\infty.
\label{smallness}
\end{equation}
Let $s_k$ be the  unique root of $\sum_n\frac{a_n\mu^{1/2}_n}{z-t_n}$
in $(t_{n_k-1},t_{n_k})$. Then
\begin{equation}
\sum_{n<m_k,\, n\in\mathcal N}\frac{a_n\mu^{1/2}_n}{s_k-t_n}
+\sum_{l=m_k}^{n_k-1}\frac{a_l\mu^{1/2}_l}{s_k-t_l}+
\frac{\mu^{1/2}_{n_k}}{k^2(s_k-t_{n_k})}+
\sum_{n>n_k,\, n\in\mathcal N}\frac{a_n\mu^{1/2}_n}{s_k-t_n}=0.
\label{st24}
\end{equation}
From this we conclude that $\Delta_{n_k}=t_{n_k}-s_k$ satisfies the estimate
\begin{equation}
  \frac{\mu^{1/2}_{n_k}}{k^2\Delta_{n_k}}\geq \frac{1}{t_{n_k}}
  \sum_{l=m_k}^{n_k-1}a_l\mu^{1/2}_l+o\Bigl(\frac1{kt_{n_k}}\Bigr)=
  \frac{1+o(1)}{kt_{n_k}}\biggl{(}\sum_{l=m_k}^{n_k-1}\mu_l\biggr{)}^{1/2},\qquad k\to\infty.
\label{st24a}
\end{equation}
Using \eqref{nkmk} we obtain that $\Delta_{n_k}=o(t_{n_k})$, and applying \eqref{st24} again and the estimate $\inf_n\frac{d_n}{t_n}>0$ we get
the estimate reverse to \eqref{st24a}. Hence,
$\Delta_{n_k}\asymp\frac{\mu^{1/2}_{n_k}}{k}\cdot\biggl{(}
\sum_{l=m_k}^{n_k-1}\mu_l\biggr{)}^{-{1/2}}\cdot t_{n_k}$.
It remains to find a positive sequence $\{b_n\}_{n\in\mathcal N}\in\ell^2$
such that $\sum_{n\in\mathcal N}\frac{a_nb_n}{z-t_n}$ vanishes at the points $s_k$.
We fix small positive $b_n$, $n\notin\mathcal{N}_0=\{n_k\}$, so that
$\sum_{n\in\mathcal{N}\setminus \mathcal{N}_0}b_n<\infty$.
The parameters $b_{n_k}$ should satisfy the system of equations
\begin{equation}
\frac{-a_{n_k}b_{n_k}}{\Delta_{n_k}}+\sum_{j\neq k}
\frac{a_{n_j}b_{n_j}}{s_k-t_{n_j}}+\sum_{n\in\mathcal{N}\setminus \mathcal{N}_0}\frac{a_nb_n}{s_k-t_n}=0.
\label{bsystem}
\end{equation}
Put $u_k=k^2\Delta_{n_k}\sum_{n\in\mathcal{N}\setminus \mathcal{N}_0}\frac{a_nb_n}{s_k-t_n}$
and rewrite our system as
$$
  b_{n_k}-k^2\Delta_{n_k}\sum_{j\neq k}\frac{1}{j^2(s_k-t_{n_j})}b_{n_j}=u_k.
$$
Let us show that $\{u_k\} \in \ell^2$. Indeed,
$\sum_{n\in\mathcal{N}\setminus \mathcal{N}_0}\frac{a_nb_n}{s_k-t_n}\lesssim \frac{1}{t_{n_k}}$
and
$$
|u_k|\lesssim \frac{k^2\Delta_{n_k}}{t_{n_k}}\asymp k\mu^{1/2}_{n_k}
\cdot\bigg(\sum_{l=m_k}^{n_k-1}\mu_l \bigg)^{-{1/2}}\lesssim\frac{1}{k}.
$$
Next consider the coefficients at $b_{n_j}$. If $j>k$,
then $t_{n_j}\gg s_k\gg k^2\Delta_{n_k}$. On the other hand,
if $j<k$, then $\frac{k^2\Delta_{n_k}}{|s_k-t_{n_j}|}\lesssim\frac{1}{k^2}$.
So, we can assume that
$$
\sum_kk^2\Delta_{n_k}\sum_{j\neq k}\frac{1}{j^2(s_k-t_{n_j})}<\frac{1}{100}.
$$
This means that we can find a summable solution
$\{b_{n_k}\}$ of equations \eqref{bsystem}.
Clearly, $\{a_n\}$ and $\{b_n\}$ satisfy conditions (a)--(d).

The case when there exists a sequence of indices
$\{n_k\}$ such that $\mu_{n_k}=o( t^2_{n_k}\sum_{l>n_k}\frac{\mu_l}{t^2_l})$, $k\to\infty$,
can be treated analogously.
\qed
\bigskip


\section{Proof of Proposition \ref{cod}}
\label{proofcod}

First, note that any function $F\in\he$ is of the form
$F(z)= A(z) \sum_n \frac{a_n\mu_n^{1/2}}{z-t_n}$, where $a= \{a_n\}\in\ell^2$,
whence
\begin{equation}
\label{cons1}
\bigg|\frac{F(iy)}{A(iy)}\bigg| \le
\|a\|_{\ell^2} \beta(y), \qquad
\beta(y) := \bigg(\sum_n \frac{\mu_n}{t_n^2+y^2}\bigg)^{1/2}.
\end{equation}

Now let the functions $h$ and $S$ in (\ref{dv}) have infinitely many common zeros
$\{s_k\}$ which admit power-type separation from $T$,
$\dist (s_k, T)\gtrsim |s_k|^{-N}$ for some $N>0$.
We will factorize $h$ and $S$ as
$h=G_2S_2$ and $S=S_1S_2$, where $S_2$ is a product over some
subsequence of common zeros, and then will construct
$G_1$ as a small perturbation of $S_2$.

Let us construct zeros $s_k$ of $S_2$ and $\tilde s_k = s_{k-1} +\delta_k$
of $G_1$ inductively. Assume that $s_1=\tilde s_1$, $\tilde s_2, s_2$, \dots,
$\tilde s_{k-1}$, $s_{k-1}$ are already constructed.
Choose $s_k$ (a common zero of $h$ and $S$) so large that if we define $\delta_k$ by
$$
\delta_k = \frac{1}{(\beta(s_k))^{1/2}}
\prod_{m=1}^{k-1}\frac{s_m}{\tilde s_m} - s_{k-1},
$$
then
$$
\tilde s_k = s_{k-1}  +\delta_k > 10 s_{k-1} \qquad
\text{and} \qquad s_k >10\tilde s_k.
$$
This is possible since $\beta(y) \to 0$, $|y| \to\infty$, but, at the same time,
$\beta(y) \gtrsim |y|^{-1}$.

Now define
$$
S_2(z) = \prod_k \bigg(1-\frac{z}{s_k}\bigg), \qquad G_1(z) = \prod_k \bigg(1-\frac{z}
{\tilde s_k}\bigg),
$$
and $G=G_1G_2$. It is clear that $|S_2(iy)| \le |G_1(iy)| \lesssim |yS_2(iy)|$, $|y|\to\infty$.

Assume first that $G \in \he+z\he$. We need to show that
$G$ is a generating function of an exact system (in particular, $G\notin \he$)
and that
\begin{equation}
\label{cons2}
\frac{G_1(z)S_1(z)}{A(z)} = \sum_n\frac{b_nG(t_n)}{\mu^{1/2}_nA'(t_n)(z-t_n)}.
\end{equation}
An obvious estimate of infinite products shows that
$$
\bigg|\frac{G_1(is_k)}{S_2(is_k)}\bigg| \asymp \prod_{m=1}^k \frac{s_m}{\tilde s_m}
\asymp s_k (\beta(s_k))^{1/2}
$$
by the construction of $\delta_k$.
By the property (d) from Subsection \ref{outline}, $|G_2(iy)S_2(iy)| =|h(iy)|
\gtrsim |y|^{-1} |A(iy)|$.
Hence,
$$
\bigg|\frac{G(is_k)}{A(is_k)}\bigg| \asymp
\bigg|\frac{G_1(is_k)}{S_2(is_k)}\bigg| \cdot \bigg|\frac{G_2(is_k)S_2(is_k)}{A(is_k)}\bigg|
\gtrsim \frac{1}{s_k} \bigg|\frac{G_1(is_k)}{S_2(is_k)}\bigg| \asymp
(\beta(s_k))^{1/2},
$$
whence $G\notin \he$.  Furthermore,
$|G(iy)| = |h(iy)G_1(iy)|/|S_2(iy)| \gtrsim |y|^{-1}|A(iy)|$, whence
$G$ is the generating system of some exact system of kernels.
The completeness of the biorthogonal system $\{g_\lambda\}_{\lambda\in
\mathcal{Z}_G}$ follows by the same arguments as in
the proof of Proposition \ref{submu}. If $(U, \nu)$ are some other spectral data
for $\he$ and $V$ is the entire function from \eqref{bor2}, then
$V$ is at most polynomial whenever $|G(iy)| \gtrsim |y|^{-N}|A(iy)|$.
Thus $V \equiv 0$, since $V\in L^2(\nu)$ and $\nu(\rl) = \infty$
by the discussion in Subsection \ref{clar}.

To get (\ref{cons2}) note first that
$$
H(z) =
\frac{G_1(z)S_1(z)}{A(z)} - \sum_n\frac{b_nG(t_n)}{\mu^{1/2}_nA'(t_n)(z-t_n)}
$$
is an entire function, since the residues at the points $t_n$
coincide. Both summands are in the Smirnov class in $\cp$
and $\cm$. Also $|G_1(iy)S_1(iy)|\lesssim |yS(iy)|\lesssim |A(iy)|$,
while the sum in the last formula at $z=iy$ is $o(1)$, $|y|\to\infty$.
By Krein's theorem we conclude that $H$ is of zero exponential type
and, finally, a constant. However,
$$
\bigg|\frac{G_1(is_k)S_1(is_k)}{A(is_k)}\bigg| \asymp
\bigg|\frac{S(is_k)}{A(is_k)}\bigg| \cdot \bigg|\frac{G_1(is_k)}{S_2(is_k)}\bigg| \lesssim (\beta(s_k))^{1/2} \to 0, \qquad s_k\to\infty,
$$
and hence $H\equiv 0$.

Now assume that $G$ does not belong to $\he + z\he$.
Recall that $\dist (s_k, T) \gtrsim |s_k|^{-N}$ for some $N>0$.
Then a simple estimate $|G_1(t)/S_2(t)| \lesssim |t|^2\,
{\rm dist}\, (t, \{s_k\})^{-1}$ implies that
$$
|G(t_n)|=\frac{|h(t_n)G_1(t_n)|}{|S_2(t_n)|} \lesssim |h(t_n)| (|t_n|+1)^{N+2},
$$
whence
$$
\bigg\{\frac{G(t_n)}{(|t_n| +1)^{N+2} E(t_n)}\bigg\} \in L^2(\mu).
$$
Thus choosing $m$ zeros $\tilde s_1, \dots, \tilde s_m$ of $G_1$
and setting $P_m (z) = (z-\tilde s_1)\dots (z-\tilde s_m)$
we conclude by \cite[Theorem 26]{br} that
$G/ P_m \in \he$ when $m$ is sufficiently large.
Take the smallest $m$ such that $G/P_m\notin\he$, but
$G/P_{m+1}\in \he$. Then $\tilde G = G/P_m = G_2 G_1/P_m$ will be the
required generating function satisfying (\ref{cons2})
with $G_1/P_m$ in place of $G_1$.
The completeness of $\{k_\lambda\}_{\lambda \in \mathcal{Z}_{\tilde G}}$
and of its biorthogonal family
follow from the fact that $|\tilde G(iy)| \gtrsim |y|^{-m-1}|A(iy)|$,
$|y| \to \infty$.
\qed
\bigskip


\section{De Branges spaces with strong $M$-basis property as Fock-type spaces}
\label{fock}

In this section we prove Theorem \ref{fock1}.
The implication (i)$\Longrightarrow$(ii) is obvious.


\subsection{Implication (ii)$\Longrightarrow$(iii)}
We will use the invertibility of $R_{\pi/2}$.
For other values of $\theta$ the proof is analogous.
We know that $A(iz)\in\he$. So, $iT=\{it_n\}$ satisfies the Blaschke
condition $\sum_n(1+|t_n|)^{-1}<\infty$. Moreover,
$\{k_\lambda\}_{\lambda\in iT}$ is a Riesz basis of reproducing
kernels in $\he$ and hence in the corresponding model space.
Therefore, $iT$ is lacunary in both
half-planes $\mathbb{C}_+$ and $\mathbb{C}_-$ (see, e.g., \cite[Lemma D.4.4.2]{nik}).

We may assume that $A(z)=\prod_n\bigl( 1-\frac{z}{t_n} \bigr)$
and $|t_n|>2$. Furthermore,
\begin{equation}
\label{Fock}
 \frac{\pi^2}{\mu_n}=\Bigl{\|}\frac{A(z)}{z-t_n}
 \Bigr{\|}^2\asymp\Bigl\|R_{\pi/2}\Bigl(
 \frac{A(z)}{z-t_n}\Bigr)\Bigr\|^2=\pi^2\sum_k
 \frac{|A(it_k)|^2}{|A'(t_k)|^2\mu_k|it_k-t_n|^2}.
\end{equation}
Next we estimate $\frac{A(it_k)}{A'(t_k)}$, $t_k\in[2^K, 2^{K+1}]$:
$$
\biggl{|}\frac{A(it_k)}{A'(t_k)}\biggr{|}^2=2|t_k|^2
\prod_{l\neq k}\biggl{|}\frac{1-it_k/t_l}{1-t_k/t_l}\biggr{|}^2
\gtrsim |t_k|^2
$$
because  $\sup_{s>0}\#\{t_n:  2^s\leq|t_n|<2^{s+1}\}<\infty$.

From \eqref{Fock} we get
$$
 \frac{1}{\mu_n}\gtrsim \sum_k\frac{|t_k|^2}{\mu_k|t_k-it_n|^2}
 \asymp\sum_{|t_k|> |t_n|}\frac{1}{\mu_k}
 + \frac{1}{|t_n|^2}\sum_{|t_k|\leq|t_n|}\frac{|t_k|^2}{\mu_k}.
$$
So, $\frac{1}{\mu_n}\gtrsim
{\sum}_{|t_k|>|t_n|}\frac{1}{\mu_k}$.
Now we make use of an elementary lemma from \cite{BMS}:

\begin{lemma}\textup(\cite[Lemma 5.9]{BMS}\textup)
\label{simplefacts}
Let $\{c_n\}$ be a sequence of positive numbers.
\begin{itemize}
\item [(i)] If there is a constant $C$ such that $\sum_{m=1}^{n-1} c_m \le C c_n$
for $n>1$, then there is a positive constant $\delta$ such that $
c_m/c_n \ge c 2^{\delta (m-n)}$ whenever $m>n$.
\item [(ii)] If there is a constant $C$ such that $\sum_{m=n+1}^{\infty} c_m
\leq C c_n$ for every positive integer $n$, then there is a positive
constant $\delta$ such that $c_m/c_n \le c 2^{-\delta (m-n)}$
whenever $m>n$.
\end{itemize}
\end{lemma}

Applying Lemma \ref{simplefacts} we conclude that
$\mu_n\le C2^{-\delta(m-n)}\mu_m$, $m>n$, for some
small $\delta>0$. Analogously, $\frac{\mu_m}{|t_m|^2}\le C2^{-\delta(m-n)}\frac{\mu_n}{|t_n|^2}$, $m>n$.
This gives us condition \eqref{mulacunary}.


\subsection{Implication (iii)$\Longrightarrow$(i)}
For simplicity we will assume that $\{t_n\}$ is a one-sided positive sequence. Put
$$
\Omega_1=\biggl{\{}z:|z|\le  \frac{t_1+t_2}{2} \biggr{\}},\quad\Omega_n=
\biggl{\{}z: \frac{t_{n-1}+t_n}{2} < |z|\leq \frac{t_n+t_{n+1}}{2} \biggr{\}}.
$$
For $z\in\Omega_n$ we can estimate $A(z)$:
\begin{equation*}
|A(z)|\asymp\frac{|z|^{n-1}}{\prod_{k=1}^{n-1}t_k}\cdot\frac{|z-t_n|}{t_n}.
\end{equation*}
Let $\varepsilon$ be a sufficiently small positive number. Put
$$
\varphi(r)=2\log|A(ir)|+\log\mu_n,\quad r\in\cup_n[(1-\varepsilon)t_n,
(1+\varepsilon)t_n],
$$
and define $\phi(r) = \infty$ for other values of $r$.

The polynomials are dense in the Fock spaces. Since they belong to $\he$, 
we need only to verify that
$\|F\|_{L^2(\nu)} \asymp \|F\|_E$, $F\in\he$, with $\nu=e^{-\varphi(|z|)}dm(z)$.
First we check that
$\nu$
is a Carleson measure for $\he$, that is,
$\|F\|_{L^2(\nu)} \lesssim \|F\|_E$, $F\in\he$.
We 
normalize our measure to pass to an equivalent
problem for the space $\mathcal{H}(T,\mu)$ from \cite{BMS} (see also Subsection \ref{alter}):
$\nu_d(z)=|A(z)|^2\nu$.  We conclude from \cite[Theorem 1.1]{BMS}
that $\nu_d$ is a Carleson measure for $\mathcal{H}(T,\mu)$. Indeed,
$$
\nu_d(\Omega_n)\asymp\frac{|t_n|^2}{\mu_n},\quad \int_{\Omega_n}\frac{d\nu_d(z)}{|z|^2}\asymp\frac{1}{\mu_n}.
$$
The sequences $\frac{t^2_n}{\mu_n}$ and $\mu_n$ grow
exponentially and we obtain the result.

For the reader's convenience
we repeat the arguments from \cite{BMS} here.
Let $f(z) = A(z)\sum_n\frac{a_n\mu^{1/2}_n}{z-t_n}$
be an arbitrary function from $\he$, $\|f\|^2_{\he}=\pi^2\sum_n|a_n|^2$.
We have to verify that $\int_\mathbb{C}|f(z)|^2e^{-\varphi(|z|)}dm(z)
\lesssim \sum_n|a_n|^2$. First,
$$
\begin{aligned}
\int_{\Omega_n}\biggl{|}\sum_k\frac{a_k\mu^{1/2}_k}{z-t_k}\biggr{|}^2d\nu_d(z)
&\lesssim
\int_{\Omega_n}\frac{d\nu_d(z)}{|z|^2}\biggl{(}\sum_{k< n}|a_k| \mu^{1/2}_k \biggr{)}^2+
|a_n|^2+\nu_d(\Omega_n)\biggl{(}\sum_{k>n}\frac{|a_k|\mu^{1/2}_k}{t_k}\biggr{)}^2 \\
& \lesssim \frac{1}{\mu_n}\biggl{(}\sum_{k< n}|a_k| \mu^{1/2}_k \biggr{)}^2+
|a_n|^2+\frac{t^2_n}{\mu_n}\biggl{(}\sum_{k>n}\frac{|a_k|\mu^{1/2}_k}{t_k}\biggr{)}^2.
\end{aligned}
$$
Hence,
$$
\int_\mathbb{C}|f(z)|^2e^{-\varphi(|z|)}dm(z)=\sum_n\int_{\Omega_n}\biggl{|}\sum_k\frac{a_k\mu^{1/2}_k}{z-t_k}\biggr{|}^2d\nu_d(z)
$$
$$
\lesssim  \sum_n\frac{1}{\mu_n}\biggl{(}\sum_{k< n}|a_k| \mu^{1/2}_k \biggr{)}^2+\sum_n|a_n|^2+\sum_n\frac{t^2_n}{\mu_n}\biggl{(}\sum_{k>n}\frac{|a_k|\mu^{1/2}_k}{t^2_k}\biggr{)}^2.
$$
Next we show that the first summand is bounded by $C\|a\|^2$
(the estimate for the last summand can be proved analogously). Indeed, H\"older's inequality gives us that  
$$
\sum_n\frac{1}{\mu_n}\Bigl(\sum_{k< n}|a_k| \mu^{1/2}_k \Bigr)^2 \le 
\sum_n\frac{1}{\mu_n}\Bigl(\sum_{k< n}|a_k|^2 \mu^{1/2}_k \Bigr)\Bigl(\sum_{k< n} \mu^{1/2}_k \Bigr) \lesssim 
\sum_k |a_k|^2 \sum_{n>k} \frac{\mu^{1/2}_k}{\mu^{1/2}_n} \lesssim  \|a\|^2.
$$

Next we verify that $\nu$ is a reverse Carleson measure for $\he$, that is,
$\|F\|_{L^2(\nu)} \gtrsim \|F\|_E$, $F\in\he$.
We need to prove that $\int_\mathbb{C}|f(z)|^2e^{-\varphi(|z|)}dm(z)
\gtrsim \sum_n|a_n|^2$. Put
$D_n=\{z: \varepsilon t_n/2 \leq|z-t_n|\leq \varepsilon t_n\}$.  We use the inequalities
\begin{gather*}
\int_{D_n}|f(z)|^2e^{-\varphi(|z|)}dm(z)\gtrsim\frac{\varepsilon^2}{\mu_n}\int_{D_n}\biggl{|}\sum_k\frac{a_k\mu^{1/2}_k}{z-t_k}\biggr{|}^2dm(z)
\\
\geq \frac{\varepsilon^2}{\mu_n}\int_{D_n}\biggl{[}\frac{\mu_n}{t^2_n\varepsilon^2}|a_n|^2-\frac{C}{|z|^2}\biggl{(}\sum_{k<n}|a_k|\mu_k^{1/2}\biggr{)}^2-C\biggl{(}\sum_{k>n}\frac{|a_k|\mu^{1/2}_k}{t_k}\biggr{)}^2\biggr{]}dm(z)
\\
\geq \varepsilon^4\biggl{[}\frac{1}{\varepsilon^2}|a_n|^2-\frac{C}{\mu_n}\biggl{(}\sum_{k<n}|a_k|\mu^{1/2}_k\biggr{)}^2-C\frac{t_n^2}{\mu_n}\biggl{(}\sum_{k>n}\frac{|a_k|\mu^{1/2}_k}{t_k}\biggr{)}^2\biggr{]}.
\end{gather*}
Here the constant $C$ depends only on the lacunarity
constant $\inf_nt_{n+1}/ t_n$. Summing up these estimates
and using the estimates from above
for $\sum_n\frac{1}{\mu_n}\Big( \sum_{k<n}
|a_k|\mu^{1/2}_k \Big)^2$ and $\sum_n\frac{t_n^2}{\mu_n}
\Big( \sum_{k>n}\frac{|a_k|\mu^{1/2}_k}{t_k}\Big)^2$
we get
$$
\begin{aligned}
\int_\mathbb{C}|f(z)|^2e^{-\varphi(|z|)}dm(z) &
\geq\sum_n\int_{D_n}|f(z)|^2e^{-\varphi(|z|)}dm(z) \\
& \geq
\varepsilon^4\biggl{[}\frac{1}{\varepsilon^2}\sum_n|a_n|^2-C\sum_n|a_n|^2\biggr{]}.
\end{aligned}
$$
If $\varepsilon$ is sufficiently small, we obtain the result.

Finally, we may redefine $\phi(r)$
for $r\notin\cup_n[(1-\varepsilon)t_n,(1+\varepsilon)t_n]$ to be  finite, but
very large, and still have the comparability of the norms.
\qed
\bigskip


\section{$M$-bases of reproducing kernels with infinite defect}
\label{infdim}

In this section we prove Theorems \ref{mdim} and \ref{indim}. 
It should be emphasized that the construction of
$M$-bases which are not strong in Theorem \ref{mainn} is done in
the reverse order: first we produce a vector in $\he$ with some
special properties and then we construct the corresponding system
$\Lambda$. This method does not answer the question about the size
of the possible defect of a given nonhereditarily complete system.

We will need the following Proposition which slightly extends (and
whose proof simplifies the proof of) \cite[Theorem 1.1]{bb}.

\begin{proposition}
\label{bior}
Let $\he$ be a de Branges space with the spectral data $(T, \mu)$
and let $\sum_n t_n^{2N-2}\mu_n<\infty$ for some $N\in \mathbb{N}$.
Then there exists an exact system
$\{k_\lambda\}_{\lambda\in \Lambda}$ of reproducing kernels in $\he$
such that its generating function $G$ satisfies $|G(iy)| \gtrsim |y|^{-N}|A(iy)|$,
its biorthogonal system is incomplete and
${\rm dim}\, \{g_\lambda: \lambda \in \Lambda\}^\perp = N$.
\end{proposition}

\begin{proof}
Recall that every function $f\in \he$ has the expansion $f(z) = A(z)\sum_n
\frac{c_n \mu_n^{1/2}}{z-t_n}$, $\{c_n\} \in \ell^2$. Since
$\sum_n t_n^{2N-2}\mu_n<\infty$, the mappings
$\phi_l(f) = \sum_n c_n t_n^l \mu_n^{1/2}$, $0\le\ell<N$,
are bounded linear functionals on $\he$. 
Furthermore, 
$$
\phi_l(f) = 0, \quad l=0, 1, \dots, N-1 \Longleftrightarrow
\lim_{y\to+\infty} y^{l+1} \frac{f(iy)}{A(iy)}=0,
\quad l=0, 1, \dots, N-1.
$$

Let $\{a_n\}\in \ell^2$ be such that $\{a_nt_n\}\notin\ell^2$,
$\sum_n a_n t_n^l\mu_n^{1/2} = 0$ for $l=0, 1, \dots N-1$,
$a_nt_n^N > 0$ for all except a finite number of $n$ and either the series
$\sum_n a_n t_n^N \mu_n^{1/2}$ diverges to $+\infty$ or it converges and
$\sum_n a_n t_n^N \mu_n^{1/2} \ne 0$. Put
$$
G(z) = A(z) \sum_n \frac{a_n t_n \mu_n^{1/2}}{z-t_n}.
$$
Then we have
$$
\lim_{y\to+\infty} y^{l} \frac{G(iy)}{A(iy)}=0,
\quad l= 0, 1, \dots N-1, \qquad
\liminf_{|y|\to +\infty} y^{N} \Bigl|\frac{G(iy)}{A(iy)}\Bigr| > 0.
$$

Since $\{a_n\}\in \ell^2$, but $\{a_nt_n\}\notin\ell^2$, the function $G$
does not belong to $\he$, but $g_\lambda(z) = \frac{G(z)}{z-\lambda} \in \he$.
Without loss of generality (changing slightly the coefficients $a_n$ if necessary)
we may assume that all zeros of $G$ are simple and put $\Lambda = \mathcal{Z}_G$
(note that $\Lambda \cap T = \emptyset$ if $a_n \ne 0$). By construction we have
$\phi_l(g_\lambda) = \lim_{y\to+\infty} y^{l} \frac{G(iy)}{A(iy)}=0$,
$l= 0, 1, \dots N-1$, for any $\lambda\in \Lambda$.
The functionals $\phi_l$ clearly are linearly independent, whence
${\rm dim}\, \{g_\lambda: \lambda \in \Lambda\}^\perp \ge N$.

On the other hand, the set of functions $f(z) =
A(z)\sum_n \frac{a_n \mu_n^{1/2}}{z-t_n}$, $\{a_n\} \in \ell^2$,
which are orthogonal to $\{g_\lambda\}_{\lambda\in \Lambda}$
is parametrized by entire functions $S$ such that
$$
\frac{G(z)S(z)}{A(z)} = \sum_{n}\frac{a_n G(t_n)}{\mu^{1/2}_n A'(t_n)(z-t_n)}
$$
(see Subsection \ref{reduc} or \cite[Section 2]{bb}). Since $S$ is of zero
exponential type (for the generating function of an exact system of reproducing kernels
the functions $G/E$ and $G^*/E$ are in the Smirnov class in $\cp$) and
$|S(iy)| = o(|A(iy)/G(iy)|)$, $|y| \to \infty$, we conclude that $S$ a polynomial
of degree at most $N-1$. Hence,
${\rm dim}\, \{g_\lambda: \lambda \in \Lambda\}^\perp \le N$.
\end{proof}
\medskip

\begin{proof}[Proof of Theorem \ref{mdim}]
1. Let $T^\circ = \{t_{n_k}\}$, $\mu^\circ = \mu|_{T^\circ}$, and let
$\mathcal{H}(E^\circ)$ be the corresponding de Branges space.
By Proposition \ref{bior}
there exists an exact system of
reproducing kernels $\{k^\circ_\lambda\}_{\lambda\in \Lambda^\circ}$
in $\mathcal{H}(E^\circ)$ whose generating function $G^\circ$ satisfies
$|G^\circ(iy)|\gtrsim |y|^{-N}|A^\circ(iy)|$ for some $N$.
Then, by Proposition \ref{submu},
there exists an $M$-basis of reproducing kernels in $\he$
which is not a strong $M$-basis. By Remark  \ref{imp1} the defect
of the partition $\Lambda = \Lambda^\circ \cup \tilde T$
(where $\tilde T = T\setminus T^\circ$) is of dimension exactly $N$.

2. If $T$ is power separated, then $|t_n|\gtrsim |n|^\rho$
for some $\rho>0$. Hence, (ii) implies that for some sequence of indices $\{n_k\}$ and for every $N\in \mathbb N$ we have
$\sum_k t_{n_k}^{2N-2} \mu_{n_k}<\infty$, and the implication (ii)$\Longrightarrow$(i)
is already contained in Statement 1.

To prove the converse, assume that $\mu_n \gtrsim |t_n|^{-M}$
for some $M>0$. By \eqref{use}, the space $\he$
is of tempered growth. Then by \cite[Theorem 5.3]{bbb}
there exists $M=M(N)$ such that
for any exact system of reproducing kernels $\{k_\lambda\}_{\lambda\in \Lambda}$
and for any partition $\Lambda = \Lambda_1\cup \Lambda_2$, $\Lambda_1\cap \Lambda_2=\emptyset$
one has ${\rm def}\,(\Lambda_1, \Lambda_2) \le M$.
\end{proof}


In the proof of Theorem \ref{indim} we use the following 
technical lemma. As usual, for a sequence $\Gamma$, we denote
by $n_\Gamma$ its counting function: $n_\Gamma(r) = 
\#\{\gamma \in \Gamma: |\gamma| <r\}$. For an entire function $f$ we write $n_f$
in place of $n_{\mathcal{Z}(f)}$. 

\begin{lemma}
\label{newl}
Let $\Gamma$ be a lacunary sequence on $\mathbb R_+$ and let $f$
be an entire function of zero exponential type such that 
\begin{equation}
\label{qn1}
\int_0^R\frac{n_f(r)}{r}dr = o\Bigl(\int_0^R  \frac{n_\Gamma(r)}{r}dr\Bigr),\qquad  R\to \infty
\end{equation} 
and $f\in\ell^\infty(\Gamma)$. Then $f$ is a constant. 
\end{lemma}

\begin{proof} Let $\Gamma = \{x_n\}_{n\in \mathbb{N}}$ and $x_{n+1}/x_n \ge e^\gamma >1$, $n\ge 1$. Passing to a subsequence, we can assume that $x_{n+1}/x_n \ge e^2$, $x_1\ge e^2$  
and \eqref{qn1} is still valid.
Set
$$
\Omega_n = \big\{z: e^{2n} \le |z| < e^{2(n+1)} \big\},
\qquad p_n = \#(\mathcal{Z}_f \cap \Omega_n),
$$
and
$$
\mathcal{N} = \big\{k\ge 1: \sum_{n>1}e^{-n}p_{k+n}\ge 1\big\}\cup
\big\{ k\ge 2: p_{k-1} +p_k + p_{k+1} >0\big\}.
$$
Finally, let
$$
U(x) = \int_0^{e^{2x}} \frac{n_\Gamma(r)}{r}dr=\int_0^{2x} n_\Gamma(e^s)\,ds.
$$
Since $n_\Gamma(r) \to \infty$, $r\to \infty$, and 
$\#(\Gamma\cap \Omega_k) \le 1$, we have
$x = o(U(x))$,  
$x\to \infty$.

Given $x\in \Gamma\cap \Omega_k$, $k\notin \mathcal{N}$, we have for some $c<\infty$ 
$$
\begin{aligned}
|f(x)| & = \prod_{z_n \in \mathcal{Z}_f} \bigg|1-\frac{x}{z_n}\bigg| \ge
\prod_{|z_n| < e^{2(k-1)}} \bigg(\frac{e^{2k}}{|z_n|} -1 \bigg)
\cdot
\prod_{|z_n| \ge e^{2(k+2)}}\bigg(1- \frac{e^{2(k+1)}}{|z_n|}\bigg) \\
& \ge \exp\Big[ \#\big\{z_n: |z_n| < e^{2(k-1)}\big\}
- \sum_{n>1} e^{-(n-1)}p_{k+n}  \Big] \to \infty, \qquad x\to \infty.
\end{aligned}
$$
(We use here that $1-a^2>\exp(-a)$ for $0<a<1/e$). Thus, the set $\Gamma\setminus 
\cup_{k\in \mathcal{N}} \Omega_k$  is finite. 
Hence, 
$$
\begin{aligned}
n_\Gamma(e^{2(k+1)}) & \le 
\#\{l: 2\le l\le k, \ p_{l-1}+p_l+p_{l+1} >0 \}  \\
& + \#\{l: 1\le l\le k, \ \sum_{n>1} e^{-n}p_{l+n} \ge 1\} +O(1) \\
& \le
n_f\big(e^{2k}\big) + n_f\big(e^{2(k+1)}\big) +
n_f\big(e^{2(k+2)}\big) + \sum_{1\le l\le k} \sum_{n>1} e^{-n}p_{l+n} +O(1)\\
& \le 3n_f\big(e^{2(k+2)}\big) +\sum_{n>1} e^{-n}n_f
\big(e^{2(k+n+1)}e^{-n}\big) +O(1), \qquad k\to\infty.
\end{aligned}
$$
Set
$$
V(x) = \int_0^{e^{2x}} \frac{n_f(r)}{r}dr=\int_0^{2x} n_f(e^s)\,ds+O(1),\qquad x\to\infty.
$$
By the hypothesis, $V(x) = o(U(x))$, $x\to \infty$.
Then, integrating the previous inequality, we get 
$$
U(k) \le 3V(k+2) +\sum_{n>1}e^{-n}V(k+n+1) +O(1),\qquad k\to\infty.
$$
Since $n_\Gamma(r) \le \frac 12 \ln r$ for large $r$, we have
$$
U(k+n) =U(k) + \int_{2k}^{2(k+n)} 
n_\Gamma(e^r)\,dr \le U(k) + (2k+n)n.
$$
Since for any $\vep>0$ we have $V(k+n)\le \vep U(k+n)$ 
for sufficiently large $k$, we conclude that
$$
U(k) \le 4\vep U(k) +c\vep k  +O(1),\qquad k\to\infty,
$$
for an absolute constant $c$.
Recall that $\vep$ is an arbitrary positive number and $k=o(U(k))$, $k\to \infty$.
Thus, we come to a contradiction.
\end{proof}

\begin{proof}[Proof of Theorem \ref{indim}]
We split the sequence $T$ into three disjoint parts $T=T^0\cup T^1\cup T^2$
with the following properties:
\smallskip

(i) Both $T^0$ and $T^1$ are positive lacunary sequences.
\smallskip

(ii) $\big[\frac{t_n}{2}, 2t_n\big] \cap T^1 = \emptyset$ for any $t_n \in T^0$,
and $\big[\frac{t_n}{2}, 2t_n\big] \cap T^0 = \emptyset$ for any $t_n \in T^1$.
\smallskip

Additionally assume that 
$$
({\rm iii})\qquad n_{T^0}(r)= o(n_{T^1}(r)),\, r\to\infty.
$$


The proof will consist of several steps. We will successively define the measure
$\mu$ on the sets $T^0$, $T^1$ and $T^2$.
\medskip


{\bf Step 1. Construction of a complete system of reproducing
kernels with incomplete biorthogonal of infinite defect.} To
define the measure $\mu$ on $T^0$, we may apply to $T^0$ the
following result proved in \cite[Theorem 4.2]{bd} (see the
beginning of the proof of Theorem 4.2 in \cite{bd}, where a 
reformulation in terms of the systems of reproducing kernels is
given):
\smallskip

{\it For any increasing sequence $T^0\subset \rl$
there exists a measure $\mu^0 = \sum_{t_n\in T^0} \mu_n\delta_{t_n}$
such that the de Branges space $\mathcal{H}(E^0)$, $E^0=A^0-iB^0$,
with the spectral data $(T^0, \mu^0)$ contains an exact system
of reproducing kernels with the generating function $G^0$,
whose biorthogonal system has
infinite-dimensional orthogonal complement.}
\smallskip

In \cite[Theorem 4.2]{bd} it is essential that the space $\mathcal{H}(E^0)$ had
some additional properties, and its proof is rather involved. Below (Step 2)
we present an explicit construction of $\mu^0$ and $G^0$ using a simplified version
of the construction from \cite[Theorem 4.2]{bd}.

It follows both from the construction in \cite{bd} and from our arguments
in Step 2 that $\mu_n \le 1$, $t_n\in T^0$, and, furthermore, that
\begin{equation}
\label{arb3}
\bigg|\frac{A^0(z)}{G^0(z)}\bigg| \le C_2|W(z)|
\bigg(\frac{|z|^2+1}{|\ima z|}\bigg)^2,
\qquad z \notin\RR,
\end{equation}
for some entire function $W$ of zero exponential type and with real zeros, such that
$n_W(r) = o(n_{T^0}(r))$, $r\to \infty$ (see \cite[inequality (4.4)]{bd}).
\medskip


{\bf Step 2. An explicit construction of $\mu^0$ and $G^0$.}
Let $A^0$ be the canonical product of zero genus with the zero set $T^0$. 
Then we may choose
a zero genus canonical product $D$ with lacunary real zeros such that 
$\mathcal Z_D\cap T^0=\emptyset$, for any $N>0$,
$$
|D(t_n)|\lesssim |t_n|^{-N}|(A^0)'(t_n)|, \qquad t_n \in T^0,
$$
and $|D(iy)| \lesssim |y|^{-N}|A^0(iy)|$ on one hand, and
$$
|A^0(z)/D(z)|\lesssim |W(z)|, \qquad {\rm dist}\, (z, \mathcal{Z}_D)\ge 1,
$$
where $W$ is an entire function satisfying
$n_W(r) =o(n_{T^0}(r))$, $r\to \infty$, on the other hand.
Put
\begin{gather*}
G^0(z)=A^0(z)\sum_{t_n\in T^0}\frac{t_nD(t_n)}{(A^0)'(t_n)(z-t_n)}=zD(z),\qquad z\in\mathbb C,\\
\mu_n=\frac{n^2 |D(t_n)|^2}{|(A^0)'(t_n)|^2}\lesssim 1, \qquad t_n \in T^0,
\end{gather*}
and define the de Branges space $\mathcal{H}(E^0)$, $E^0=A^0-iB^0$,
with the spectral data $(T^0, \mu^0)$, $\mu^0=(\mu_n)$.
Then it is easy to see that
$G^0\notin \mathcal{H}(E^0)$, $g^0_\lambda=G^0/(\cdot-\lambda)\in
\mathcal{H}(E^0)$, $\lambda\in\Lambda^0$, where $\Lambda^0$ is the zero set of $G^0$.
Furthermore, ${\rm dim}\, \{g^0_\lambda: \lambda \in \Lambda^0\}^\perp =\infty$
because
for the bounded functionals $\psi_l(f) = \lim_{y\to+\infty} y^{l+1} f(iy)/A^0(iy)$,
$l\ge 0$ (see the proof of Proposition~\ref{bior}), on $\mathcal{H}(E^0)$ 
we have $\psi_l(g^0_\lambda) = 0$,
$\lambda\in\Lambda$, by the choice of the function $D$.

It remains to prove that $G^0$ is the generating function
of a complete system of reproducing kernels in $\mathcal{H}(E^0)$.
If it is not the case, then there exists an entire function $H$
such that $HG^0\in\mathcal{H}(E^0)$. Hence, $H$
is of zero exponential type, its growth is majorized by that of $A^0/D$ away from
the zeros of $A^0$ and $D$, and, thus, by the growth of the entire function $W$, 
i.e., $n_{H}(r)=  O(n_{W}(r)) = o(n_{T^0}(r))$, $r\to \infty$.
Finally,
$$
\sum_{t_n\in T^0}\frac{|t_nD(t_n)H(t_n)|^2}{\mu_n |(A^0)'(t_n)|^2}=
\sum_{t_n\in T^0}\frac{t^2_n|H(t_n)|^2}{n^2}<\infty,
$$
whence, $H(t_n)\to 0$, $t_n \in T^0$, $n\to\infty$.
Applying Lemma \ref{newl} to $\Gamma = T^0$ and $f=H$ 
we conclude that $H\equiv 0$.
\medskip


{\bf Step 3. Defining the measure $\mu$  on $T^1$.}
Let $\mu^1$ be defined on $T^1$ by $\mu_n =1$, $t_n \in T^1$. Consider the de Branges space $\mathcal{H}(\tilde E)$
with the spectral data $(T^0\cup T^1, \mu^0+\mu^1)$.
Let $(\tilde T, \tilde \mu)$ be some other spectral data
for $\mathcal{H}(\tilde E)$ sufficiently close to $(T^0\cup T^1, \mu^0+\mu^1)$.
We claim that
$$
\tilde \mu\big([t_n -1, t_n + 1]\big)
\gtrsim 1, \qquad t_n \in T^1.
$$
Indeed, let $t_n\in T^1$. It follows from simple estimates 
of the derivative of the inner function
$\Theta_{\tilde E}$ based on formula \eqref{use}
that $|\Theta_{\tilde E}'(t)| \asymp 1 $, $|t-t_n| \le 2$.
Hence, if the spectral data
$(\tilde T, \tilde \mu)$ correspond to some $\alpha$ sufficiently close to $-1$
(recall that we assume canonically that $(T^0\cup T^1, \mu^0+\mu^1)$ correspond
to $\alpha= -1$), then there exists a point $\tilde t_n \in \tilde T$
in the interval $[t_n - 1, t_n+1]$ such that $|\Theta'_{\tilde E}(\tilde t_n)| \asymp 1$.
\medskip


{\bf Step 4. Defining the measure $\mu$  on $T^2$.} We can
choose the spectral data $(\tilde T, \tilde \mu)$ in Step 3 so
that additionally $\tilde T \cap T^2 = \emptyset$. Now we take $\mu_n$ for
$t_n\in T^2$ to be very small positive numbers. Namely, if we
denote by $\he$ the de Branges space with the spectral data $(T,
\mu)$ and by $\Theta_E$ the corresponding inner function, we need
to choose $\mu_n$ for $t_n\in T^2$ so small that for some other
spectral data $(U, \nu)$ for $\he$ 
there exists a point $u_n \in [\tilde t_n-1, \tilde t_n+1] \cap U$
with $|\Theta_E'(u_n)| \asymp 1$. This is clearly possible since
adding to $\mu$ a small point mass at some point $x$ perturbs very slightly the
solutions of the equation $\Theta_E(t) = \beta \in \mathbb{T}$ and
the derivative $|\Theta_E'(t)|$ outside a small neighborhood of $x$.

Thus, we have constructed a de Branges space $\he$
with the spectral data $(T, \mu)$ such
that for some other spectral data $(U, \nu)$ for $\he$
we have $\nu([t_n-2, t_n+2]) \gtrsim 1$, $t_n\in T^1$
(recall that $\nu(\{u_n\}) = 2\pi/|\Theta_E'(u_n)|\asymp 1$).
If $E=A-iB$, then we can write $A=A^0 A^1 A^2$, where $A^0$, $A^1$
are canonical products of zero genus with the zero sets $T^0$ and $T^1$, respectively, and
$A^2$ is some entire functions with the zero set $T^2$.
\medskip


{\bf Step 5. Construction of the set $\Lambda$.}
By construction, the space $\mathcal{H}(E^0)$
contains an exact system of reproducing kernels
$\{k^0_\lambda\}_{\lambda\in \Lambda^0}$
with the generating function $G^0$
whose biorthogonal system has infinite-dimensional orthogonal complement,
and, thus, there exists an infinite-dimensional subspace of vectors
$a^0 = \{a_n^0\} \in \ell^2$ such that
$$
\frac{G_1^0(z)S^0_1(z)}{A^0(z)}  =
\sum_{t_n\in T^0} \frac{a^0_n G^0(t_n)}{\mu^\half_n (A^0)'(t_n)(z-t_n)},
\qquad
\frac{S^0_2(z)}{A^0(z)}   =
\sum_{t_n\in T^0}\frac{\overline{a_n^0}\mu^\half_n}{z-t_n}
$$
for some entire functions $S^0_1$ and $S^0_2$ depending on $a^0$.
Multiplying these equations
by $A^1 A^2$ as in the proof of Proposition \ref{submu},
we conclude that the system $\{k_\lambda\}_{\lambda\in T^1\cup T^2}
\cup\{g_\lambda\}_{\lambda\in \Lambda^0}$
has an infinite-dimensional orthogonal complement in $\he$.
Put $\Lambda = \Lambda^0 \cup T^1 \cup T^2$.
It is clear that the system $\{k_\lambda\}_{\lambda\in \Lambda}$
is an exact system of reproducing kernels in $\he$
(see Remark \ref{imp2}) and $G = G^0 A^1 A^2$ is its generating function.
\medskip


{\bf Step 6. Completeness of the biorthogonal system.}
It remains to show that the system $\{G/(\cdot-\lambda)\}_{\lambda\in \Lambda}$
biorthogonal to $\{k_\lambda\}_{\lambda\in \Lambda}$
is complete in $\he$. We cannot apply Proposition \ref{submu}
since we do not have the condition
$|G^0(iy)|\gtrsim |y|^{-N}|A^0(iy)|$ for some $N>0$
(it is a crucial property of the construction in \cite{bd}
or in Step 2 that $G^0/A^0$ decreases super-polynomially along the imaginary axis).
We apply another argument from \cite{bd}.

Let $(U, \nu)$, $U=\{u_n\}$, be the spectral data for $\he$
constructed above. They correspond to the function $E_\alpha =
\alpha E-E^*$ for some $\alpha\in \tz$, $\alpha \ne -1$. By the
arguments from Subsection \ref{reduc}, the system
$\{G/(\cdot-\lambda)\}_{\lambda\in \Lambda}$ is incomplete in
$\he$ if and only if there exists a sequence $\{c_n\} \in \ell^2$
and a nonzero entire function $V$ such that
$$
\frac{G(z)V(z)}{E_\alpha(z)} = \sum_{n}\frac{c_n G(u_n)}
{\nu^\half_n E_\alpha'(u_n)(z-u_n)},
$$
where $V(u_n) = \nu_n^{-1/2} c_n$ and $V\in L^2(\nu)$.
Since
$$
\frac{G}{E_\alpha} = \frac{G^0}{A^0}\cdot \frac{A}{E_\alpha} =
\frac{G^0}{A^0}\cdot \frac{1+\Theta_E}{\Theta_E-\alpha}
$$
and $A^0/G^0$ satisfies \eqref{arb3},
we conclude that $|V(z)|\lesssim
(|z|+1)^N |W(z)|/|\ima z|^M$, for some
$M, N>0$.
Here we used the fact that $1-|\Theta_E(z)| \ge (1+|z|^2)^{-1} \ima z$, 
$z\in \mathbb{C}_+$. 
Now it follows from standard estimates
based on the Jensen formula (see, e.g., \cite[Section 4]{bd} for details)
that the counting function $n_V$ satisfies
$$
\int_0^R  \frac{n_V(r)}{r}dr =
o\bigg(\int_0^R  \frac{n_{W}(r)}{r}dr\bigg), \qquad R\to \infty.
$$
Since $n_W(r) = o(n_{T^0}(r))$ and 
$T^0$ is lacunary, we conclude, using (iii), that 
$n_V(r) = o(n_{T^1}(r))$, $r\to\infty$. 

Also, since $V\in L^2(\nu)$ and
$\nu([t_n -2, t_n +2]) \gtrsim 1$, $t_n\in T^1$,
we have $\inf_{[t_n -2, t_n+2]} |V| \lesssim 1$. Thus, 
by Lemma \ref{newl}, $V$ is a constant,
and, finally, $V\equiv 0$, since $\nu(\rl) = \infty$. This proves the
completeness of the biorthogonal system.
\end{proof}
\bigskip


\end{document}